\documentclass{amsproc}

\usepackage{amssymb}

\usepackage[cmtip,all]{xy}

\usepackage{amsrefs}
\usepackage{mathtools}
\usepackage{mathrsfs}
\usepackage{tikz-cd}
\usepackage[hidelinks]{hyperref}

\newcommand{\Z}{\mathbb{Z}}
\newcommand{\Q}{\mathbb{Q}}

\newcommand{\F}{\mathbb{F}}

\newcommand{\m}{\mathfrak{m}}

\newcommand{\sO}{\mathcal{O}}
\newcommand{\wt}{\widetilde}

\newcommand{\sht}{\mathrm{ht}}
\newcommand{\fI}[2]{I_{#1}\!\left(#2\right)}

\DeclareMathOperator{\Spec}{Spec}
\DeclareMathOperator{\Proj}{Proj}

\DeclareMathOperator{\Hom}{Hom}

\DeclareMathOperator{\Exc}{Exc}
\DeclareMathOperator{\Ker}{Ker}
\DeclareMathOperator{\lct}{lct}
\DeclareMathOperator{\fpt}{fpt}
\DeclareMathOperator{\qfpt}{qfpt}
\DeclareMathOperator{\Image}{Im}

\newtheorem{theorem}{Theorem}[section]
\newtheorem{lemma}[theorem]{Lemma}
\newtheorem{proposition}[theorem]{Proposition}
\newtheorem{corollary}[theorem]{Corollary}
\newtheorem{claim}[theorem]{Claim}
\newtheorem*{claim*}{Claim}
 \newtheorem*{theoremA}{Theorem A}
 \newtheorem*{theoremB}{Theorem B}
 \newtheorem*{theoremC}{Theorem C}

\theoremstyle{definition}
\newtheorem{definition}[theorem]{Definition}
\newtheorem{example}[theorem]{Example}

\newtheorem*{definition_Intro}{Definition}

\theoremstyle{remark}
\newtheorem{remark}[theorem]{Remark}
\newtheorem{setting}[theorem]{Setting}

\numberwithin{equation}{section}

\begin{document}

\title{Quasi-$F$-singularities and singularities in birational geometry}


\author{Tatsuro Kawakami}
\address{Graduate School of Mathematical Sciences, University of Tokyo, 3-8-1 Komaba, Meguro-ku, Tokyo 153-8914, Japan}
\email{kawakami@ms.u-tokyo.ac.jp}

\author{Shunsuke Takagi}
\address{Graduate School of Mathematical Sciences, University of Tokyo, 3-8-1 Komaba, Meguro-ku, Tokyo 153-8914, Japan}
\email{stakagi@ms.u-tokyo.ac.jp}

\author{Shou Yoshikawa}
\address{Institute of Science, Tokyo 152-8551, Japan}
\email{yoshikawa.s.9fe9@m.isct.ac.jp}

%    The 2020 edition of the Mathematics Subject Classification is
%    the current definitive version.
\subjclass[2020]{14B05, 13A35, 14J17}

\date{}

\begin{abstract}
We give an overview of the theory of quasi-$F$-singularities, focusing on their connection with singularities in birational geometry.
\end{abstract}

\maketitle

\section{Introduction}
$F$-singularities are singularities in positive characteristic defined in terms of the Frobenius map and include four major classes: $F$-regular, $F$-pure, $F$-rational and $F$-injective singularities. 
For the purpose of this introduction, let $(A,\m)$ be the local ring at a closed point of a $d$-dimensional normal variety over a perfect field of characteristic $p>0$. 
Then the most basic class of $F$-singularities, $F$-purity (also known as local $F$-splitting), is defined by the splitting of the Frobenius map $F \colon A \to F_*A$. 
When $A$ is Gorenstein, this is equivalent by Matlis duality to the injectivity of the induced map on local cohomology $H^{d}_{\m}(A) \to H^{d}_{\m}(F_*A)$. 
Although $F$-singularities have their origins in commutative algebra, they have turned out to be closely connected to singularities in birational geometry. 
For instance, by \cite{Smith97}, \cite{hw02} and \cite{Schwede09_2}, $F$-regular (resp.~$F$-pure, $F$-rational, Cohen--Macaulay $F$-injective) singularities are klt (resp.~lc, pseudo-rational, pseudo-Du~Bois):
\vspace*{2em}
\[
\xymatrix{
\textup{$F$-regular} \ar@{=>}[r] \ar@{=>}[d] \ar@{=>}@/^1.5pc/[rrr] & \textup{$F$-pure} \ar@{=>}[d] \ar@{=>}@/^1.5pc/[rrr] & & \textup{klt} \ar@{=>}[r] \ar@{=>}[d] & \textup{lc} \ar@{=>}[d]  \\
\textup{$F$-rational} \ar@{=>}[r] \ar@{=>}@/^1.5pc/[rrr] &  \textup{$F$-injective} \ar@{=>}@/^1.5pc/[rrr]^{\hspace*{6em}+\textup{CM}} & & \textup{pseudo-rational} \ar@{=>}[r] &  \textup{pseudo-Du~Bois}\\
}
\]
(see Definition~\ref{def:p-DB} for the definitions of pseudo-rational and pseudo-Du~Bois singularities). 
These connections have made $F$-singularities increasingly important in birational geometry. 
Quasi-$F$-singularities are a generalization of $F$-singularities, and in this paper we survey their relations to singularities in birational geometry. 

The study of quasi-$F$-singularities was initiated by Yobuko \cite{yobuko19}.  
He introduced the notion of quasi-$F$-splitting, a generalization of $F$-splitting, to study the liftability to characteristic zero of Calabi--Yau varieties of finite Artin--Mazur height. 
This notion has since proved useful in the study of global properties of varieties in positive characteristic. 
For instance, although Kodaira-type vanishing theorems generally fail in positive characteristic, Nakkajima--Yobuko \cite{NY21} proved the Kodaira vanishing theorem for proper smooth (globally) quasi-$F$-split varieties. This was recently generalized by Petrov to the Kodaira--Akizuki--Nakano vanishing theorem \cite{Pet25}.

The definition of quasi-$F$-splitting involves the ring $W_n A$ of Witt vectors of length $n$, 
a local ring of characteristic $p^n$ whose reduced quotient $W_nA/\sqrt{0}$ is naturally isomorphic to $A$. 
The ring $W_nA$ is endowed with the Frobenius and restriction maps $F\colon W_nA \to F_*W_nA$ and $R^{n-1}\colon W_nA \to A$ (see \S\ref{Witt ring} for basic properties of the ring $W_nA$). 
We recall the definition of local quasi-$F$-splitting. 
\begin{definition_Intro}[\cite{yobuko19}]
We say that the local ring $A$ is \emph{quasi-$F$-split} if there exist an integer $n \ge 1$ and a $W_nA$-module homomorphism $\varphi \colon F_*W_nA \to A$ such that $\varphi \circ F=R^{n-1}$. 
This is equivalent to saying that the map $\Phi_{A,n} \colon A \to Q_{A,n}$ splits as a $W_nA$-linear map, where $Q_{A,n}$ is the pushout of $F \colon W_nA \to F_*W_nA$ and $R^{n-1} \colon W_nA \to A$. 
\[
\xymatrix{
W_n A \ar[r]^F \ar[d]_{R^{n-1}} & F_*W_n A \ar@{.>}[dl]_{\varphi} \ar[d] \\
A \ar[r]^{\Phi_{A,n}} & Q_{A,n}
}
\]
When $A$ is Gorenstein, by Matlis duality this is further equivalent to the injectivity of the induced map on local cohomology $H^{d}_{\m}(A) \to H^{d}_{\m}(Q_{A,n})$.
\end{definition_Intro}

When $n=1$, the above condition is exactly $F$-purity. 
The systematic study of quasi-$F$-splitting from the viewpoints of commutative algebra and birational geometry was first undertaken in \cite{KTY22} and \cite{KTTWYY1}. 
Further developments include \cite{TWY}, \cite{KTTWYY2}, \cite{KTTWYY3} and \cite{Kaw7}. 
In particular, these papers introduced and studied other classes of quasi-$F$-singularities, namely quasi-$F^{\infty}$-split, quasi-$F$-regular, quasi-$F$-rational and quasi-$F$-injective singularities.
Quasi-$F^{\infty}$-splitting is an iterated version of quasi-$F$-splitting, and quasi-$F$-regularity, quasi-$F$-rationality and quasi-$F$-injectivity are generalizations of $F$-regularity, $F$-rationality and $F$-injectivity, respectively. 
Roughly speaking, they are defined by replacing, in the theory of $F$-singularities, the $A$-linear map 
\[H^{d}_{\m}(A) \to H^{d}_{\m}(F^e_*A),\] 
induced by the $e$-times iterated Frobenius map $F^e \colon A \to F^e_*A$, with the $W_nA$-linear map 
\[H^{d}_{\m}(A) \to H^{d}_{\m_n}(Q^e_{A,n}).\] 
Here $Q^e_{A,n}$ is the pushout of $F^e \colon W_nA \to F^e_*W_nA$ and $R^{n-1} \colon W_nA \to A$, and $\m_n$ is the maximal ideal of $W_nA$. 
The reader is referred to Definition~\ref{def:quasi-F-invariants} for the precise definitions. 

The relationships between quasi-$F$-regular and klt singularities and between quasi-$F$-split and lc singularities are studied in \cite{KTTWYY3} and \cite{STY}, respectively.
\begin{theoremA}
Suppose that $A$ is $\Q$-Gorenstein. 
\begin{enumerate}
    \item $($\cite{KTTWYY3}*{Theorems~A~and~C}$)$ If $A$ is quasi-$F$-regular, then $\Spec A$ is klt. The converse holds in dimension two. 
    \item $($\cite{STY}*{Theorems~A~and~6.18}$)$ If $A$ is quasi-$F^{\infty}$-split, then $\Spec A$ is lc. The converse holds in dimension two when the characteristic is $p>3$. 
\end{enumerate}
\end{theoremA}
Since quasi-$F^{\infty}$-splitting implies Yobuko's original notion of quasi-$F$-splitting, and the converse holds in the Gorenstein case by \cite{KTTWYY3}*{Theorem~H}, Theorem~A shows that Gorenstein quasi-$F$-split singularities are lc.

On the other hand, \cite{KTTWYY3}*{Corollary~3.15} shows that quasi-$F$-rational singularities are pseudo-rational, but it was not known whether the converse holds in dimension two. 
Furthermore, quasi-$F$-injective singularities have not been studied extensively in the literature.
In this article, we establish some basic properties of quasi-$F$-injective singularities. 
\begin{theoremB}
The following two assertions hold. 
\begin{enumerate}
\item \textup{(Proposition~\ref{cor:qFinj-to-qF^einj})} Suppose that $A$ is Cohen--Macaulay. 
Then $A$ is quasi-$F$-injective if and only if the Frobenius-induced map 
\[
\varprojlim_{n} H^d_{\m_n}(W_nA) \to \varprojlim_{n} F_* H^d_{\m_n}(W_nA)
\] 
is injective. 
\item \textup{(Theorem~\ref{inv-adj})} Quasi-$F$-rationality deforms when the ambient space is quasi-$F$-injective. 
That is, for a nonzero element $f \in \m$, if $A$ is quasi-$F$-injective and $A/(f)$ is quasi-$F$-rational, then $A$ is quasi-$F$-rational. 
\end{enumerate}
\end{theoremB}
Note that quasi-$F$-regularity, quasi-$F$-splitting, quasi-$F$-rationality and quasi-$F$-injectivity do not deform in general, even for Gorenstein rings (see Example~\ref{example:fedder-qFs}(3) and Remark~\ref{example-inv-adj}).

Using ideas from \cite{TWY}, \cite{KTTWYY3} and \cite{STY}, 
we also investigate the relationship between quasi-$F$-injective (resp.~quasi-$F$-rational) singularities and pseudo-Du~Bois (resp.~pseudo-rational) singularities.
Our new results are summarized as follows.
\begin{theoremC}
The following assertions hold. 
\begin{enumerate} 
    \item \textup{(Theorem~\ref{rational-to-qFrat})} In dimension two, pseudo-rational singularities are quasi-$F$-rational.  
    \item \textup{(Theorems~\ref{qFinj-to-Du-Bois} and \ref{thm:DuBois-to-qFinj})} If $A$ is quasi-$F$-injective, then $\Spec A$ has pseudo-Du~Bois singularities. The converse holds in dimension two if, in addition, the residue field $A/\m$ is perfect. 
\end{enumerate}
\end{theoremC}

In general, $F$-singularities are more restrictive than the corresponding singularities in birational geometry (see Example~\ref{example:Fsing}). 
However, as discussed above, in dimension two, quasi-$F$-singularities agree with them under mild hypotheses.  
This is one advantage of considering quasi-$F$-singularities as opposed to classical $F$-singularities. 
It is natural to ask to what extent such results extend to higher dimensions.
In characteristic $p>41$, three-dimensional $\mathbb Q$-factorial klt singularities over a perfect field are quasi-$F$-regular (see \cite{KTTWYY2}*{Theorem~A}), and the bound $p>41$ is sharp (see \cite{KTTWYY2}*{Theorem~B}). 
On the other hand, 
by considering the affine cone over a supersingular K3 or abelian surface, one can construct a three-dimensional lc singularity that is not quasi-$F$-split in every characteristic $p>0$. We refer the reader to Theorem~\ref{thm:quasi-F-split of 3-dim klt} and Remark~\ref{remark:higher-dimensional case} for more details. 

Since quasi-$F$-singularities are closely related to singularities in birational geometry, it is natural to ask how a given singularity can be classified in terms of quasi-$F$-singularities.
When the singularity is a hypersurface, we explain criteria due to \cite{KTY22} and \cite{Yoshikawa25-fedder} that are analogous to Fedder's criterion for $F$-singularities. 

The paper is organized as follows. 
We start with preliminary material on singularities in birational geometry and rings of Witt vectors in Section~2. 
In Section~3, we recall classical $F$-singularities and their relationships with singularities in birational geometry. 
In Section~4, we introduce quasi-$F$-singularities and prove a deformation result for quasi-$F$-rationality. 
In Section~5, we compare quasi-$F$-singularities with singularities in birational geometry, focusing especially on quasi-$F$-rationality versus pseudo-rationality and quasi-$F$-injectivity versus the pseudo-Du~Bois property.  
In Section~6, we discuss Fedder-type criteria for hypersurface singularities and quasi-$F$-pure thresholds. 
Finally, Section~7 lists some open problems on quasi-$F$-singularities. 

\subsection*{Acknowledgements}
The authors are grateful to Kenta Sato for valuable discussions and to the referees for their careful reading of the manuscript and many useful comments. 
The first, second and third authors were supported by JSPS KAKENHI Grant Nos.~JP24K16897, 23K22383 and 25H00399, and JP24K16889, respectively. 
The first author was also supported by Inamori Foundation. 

\section{Preliminaries}
\subsection{Notation}
In this subsection, we summarize the notation that will be used in this paper. 
\begin{enumerate}
\item All rings are commutative with unit and are assumed to be Noetherian and excellent.
All schemes are assumed to be Noetherian, separated and excellent.  
\item For a ring $R$, the set $R^{\circ}$ denotes the set of elements of $R$ that do not lie in any minimal prime of $R$. 
\item Let $R$ be a normal domain and $D$ be a $\Q$-Weil divisor on $X:=\Spec R$. Then $H^0(X, \sO_X(\lfloor D \rfloor ))$ is denoted by $R(D)$. 
A canonical divisor of $X$ is denoted by $K_R$ and the associated module is denoted by $\omega_R$.  
\item
Suppose that $(R,\m)$ is a local ring. 
If $f \colon M \to N$ is an $R$-module homomorphism, then by abuse of notation, we also write $f \colon H^i_{\m}(M)\to H^i_{\m}(N)$ for the induced map  on local cohomology.
\end{enumerate}

\subsection{Singularities in birational geometry}
First, we recall pseudo-rational singularities and introduce a related notion of pseudo-Du~Bois singularities.
Since we work in positive characteristic, resolutions of singularities are not known to exist in general.
\begin{definition}\label{def:p-DB}
Let $(R,\m)$ be an excellent Cohen--Macaulay normal local ring with dualizing complex, and set $X:=\Spec R$. 
\begin{enumerate}
\item[(i)] (\cite{LT81}) We say that $X$ has \textit{pseudo-rational} singularities if for every projective birational morphism $\pi:Y \to X$ from a normal integral scheme $Y$, one has $\pi_*\omega_Y=\omega_X$. 
\item[(ii)] We say that $X$ has \textit{pseudo-Du~Bois} singularities if for every projective birational morphism $\pi \colon Y \to X$ from a normal integral scheme $Y$ whose exceptional locus is a divisor $E=\sum_i E_i$, one has $\pi_*\omega_Y(E)=\omega_X$. 
\end{enumerate}
\end{definition}

\begin{remark}
The terminology ``pseudo-Du~Bois singularities'' is not standard.
The notion is intended as a characteristic-free version of Du~Bois singularities in the Cohen--Macaulay setting, in the same spirit as pseudo-rational singularities.
\end{remark}

\begin{lemma}\label{DB lem}
    Let $X$ be a normal integral scheme.
    Suppose that there exists a log resolution $\pi\colon Y\to X$ with reduced exceptional divisor $E$ such that $\pi_*\omega_Y(E)=\omega_X$.
    Then, for every projective birational morphism $f\colon Z\to X$ from a normal integral scheme $Z$ whose exceptional locus is a divisor $F=\sum_i F_i$, 
    we have $f_{*}\omega_Z(F)=\omega_X$. 
\end{lemma}

\begin{proof}
Let $W$ be the normalization of the irreducible component of $Y\times_X Z$ dominating $Y$.
We then have the following commutative diagram:
    \[
\begin{tikzcd}
W \arrow[r, "\theta"] \arrow[d,"g"]\arrow[rd,"h"]  & Y \arrow[d, "\pi"] \\
 Z \arrow[r, "f"] & X.
\end{tikzcd}    
\]
Since $(Y,E)$ is a simple normal crossing pair, there exists a natural pullback morphism
$\theta^* \colon \omega_Y(E) \to 
\theta_{*}\omega_W(\Exc(h))$ 
such that the composition 
\[
\omega_Y(E)\xrightarrow{\theta^{*}} \theta_{*}\omega_W(\Exc(h))
\hookrightarrow \omega_Y(E)
\]
is an isomorphism. Consequently,  
$\theta_{*}\omega_W(\Exc(h))=\omega_Y(E)$, 
and pushing this forward via $\pi$ yields 
\[
h_{*}\omega_W(\Exc(h))=\pi_*\omega_Y(E)=\omega_X. 
\]
It follows that the composition of natural inclusions
\[
h_{*}\omega_W(\Exc(h)) \hookrightarrow 
f_{*}\omega_Z(F) 
\hookrightarrow \omega_X
\]
is an isomorphism, and therefore $f_{*}\omega_Z(F)=\omega_X$. 
\end{proof}

In characteristic zero, pseudo-rational singularities coincide with rational singularities. 
An analogous result holds for pseudo-Du~Bois singularities.
\begin{corollary}
Let $(R,\m)$ be a Cohen--Macaulay normal local ring essentially of finite type over the field $\mathbb{C}$ of complex numbers. 
Then $\Spec R$ has pseudo-Du~Bois singularities if and only if it has Du Bois singularities. 
\end{corollary}
\begin{proof}
This follows immediately from \cite{KSS10}*{Theorem 1.1} and Lemma~\ref{DB lem}. 
\end{proof}

Next, we recall the definitions of singularities of the minimal model program in arbitrary characteristic. 
\begin{definition}
Let $(R,\m)$ be an excellent normal local ring with dualizing complex and $\Delta$ be an effective $\Q$-Weil divisor on $X:=\Spec R$ such that $K_X+\Delta$ is $\Q$-Cartier. 
\begin{enumerate}
\item[(i)] 
Given a projective birational morphism $\pi:Y \to X$ from a normal integral scheme $Y$, for each prime divisor $E$ on $Y$, the \textit{discrepancy} $a_E(X,\Delta)$ of the pair $(X,\Delta)$ at $E$ is defined as 
\[
a_E(X,\Delta)=\mathrm{ord}_E(K_Y-\pi^*(K_X+\Delta)). 
\]
It is said that $(X,\Delta)$ is \textit{klt} (resp.~\textit{lc}) if $a_E(X,\Delta)>-1$ (resp.~$a_E(X,\Delta) \ge -1$) for every projective birational morphism $\pi:Y \to X$ from a normal integral scheme $Y$ and for every prime divisor $E$ on $Y$. 
When $\Delta=0$, we simply say that $X$ has klt or lc singularities. 
\item[(ii)] 
Suppose that $X$ is $\Q$-Gorenstein and has lc singularities. 
For an effective $\Q$-Cartier $\Q$-Weil divisor $D$ on $X$,  its \textit{lc threshold} $\lct(X;D)$ is defined as 
\[
\lct(X;D)=\sup\{t \in \Q_{\ge 0} \mid (X, t D) \textup{ is lc}\}. 
\]
\end{enumerate}
\end{definition}

\subsection{Rings of Witt vectors}\label{Witt ring}
In this subsection, we briefly review some basic properties of rings of Witt vectors, which we will use later. 

Suppose that $A$ is a ring of prime characteristic $p$. 
For an integer $e \ge 1$ and an $A$-module $M$, let $F^e_*M$ denote the $A$-module whose underlying abelian group is $M$ and whose scalar multiplication is given by $r \cdot x:=r^{p^e}x$ for $r \in A$ and $x \in F^e_*M$. 
To distinguish elements of $F^e_*M$ from those of $M$, we write an element of $F^e_*M$ in the form $F^e_*x$ with $x \in M$. 
We say that $A$ is \textit{$F$-finite} if $F \colon \Spec A \to \Spec A$ is a finite morphism, or equivalently if $F_*A$ is a finitely generated $A$-module.

Given an integral domain $A$ of characteristic $p>0$ and an integer $n \ge 1$, the ring $W_n(A)$ of \textit{Witt vectors} of length $n$ is the set 
\[
W_nA=A^{\oplus n}=\{(a_0, \dots, a_{n-1}) \mid a_0, \dots, a_{n-1} \in A\}
\]
equipped with a certain ring structure. 
$W_1 A$ coincides with $A$. 
If $n \ge 2$, then the quotient ring of $W_nA$ by its nilradical is naturally isomorphic to $A$. 
The \textit{Teichm\"{u}ller lift} $[a]$ of an element $a \in A$ in $W_nA$ is defined as 
\[
[a] \coloneqq (a,0,0, \dots, 0) \in W_nA.
\]
Then $[0]$ is the zero element and $[1]$ is the identity element of $W_nA$. 
In general, multiplication by $[a]$ is defined as follows: 
\[
[a] (b_0, \dots, b_{n-1}) = (ab_0,a^pb_1, a^{p^2}b_2, \dots, a^{p^{n-1}}b_{n-1}) \quad ((b_0, \dots, b_{n-1}) \in W_nA). 
\]
If $f \colon A \to B$ is a ring homomorphism, where $B$ is an integral domain of characteristic $p>0$, then the induced map 
\[
W_nf \colon W_nA \to W_nB; \quad (a_0, \dots, a_{n-1}) \mapsto (f(a_0), \dots, f(a_{n-1}))
\]
is a ring homomorphism. 
The endomorphism $W_nF\colon W_nA \to W_nA$ will simply be denoted by $F$ and called the Frobenius map on $W_nA$. Although $W_nA$ has characteristic $p^n$ rather than $p$, we can still define $F^e_*M$ for any $W_nA$-module $M$ just as in the case of $A$-modules, by viewing $M$ as a $W_nA$-module via the $e$-times iterated Frobenius map $F^e \colon W_nA \to W_nA$.  

The ring $W_nA$ has three important $W_nA$-module homomorphisms:
\begin{alignat*}{3}
&(\textrm{Frobenius}) &&F \colon 
W_nA \longrightarrow F_*W_nA & F(a_0,a_1,\dots ,a_{n-1}) &=(a_0^p,a_1^p,\dots, a_{n-1}^p) ,  \\ 
&(\textrm{Verschiebung}) \hspace*{0.5em} && V \colon 
F_*W_{n-1}A \longrightarrow W_{n}A \quad & V(a_0,a_1,\dots ,a_{n-2})&=(0,a_0,\dots, a_{n-2}), \\
&(\textrm{Restriction}) && R \colon 
W_{n}A \longrightarrow W_{n-1}A & R(a_0,a_1,\dots ,a_{n-1}) &=(a_0,a_1,\dots, a_{n-2}).
\end{alignat*}
The restriction $R$ is also a ring homomorphism. Therefore, $W_m A$ is naturally a $W_nA$-algebra for $m \ge n$. 
Moreover, there exists a natural exact sequence 
\[
0 \to F^e_*W_{n}A \xrightarrow{V^e} W_{n+e}A \xrightarrow{R^n} W_e A \to 0.
\]
We note that if $A$ is Noetherian and $F$-finite, then $W_nA$ is Noetherian and the Frobenius map on $W_nA$ is finite for every $n \geq 1$, by the exact sequence above. 

Since the formation of $W_nA$ commutes with localization, the affine construction glues to arbitrary $\mathbb F_p$-schemes.
Thus, for an $\F_p$-scheme $X$, we obtain a scheme $W_nX:=(|X|, W_n\sO_X)$, whose dualizing sheaf, when it exists, is denoted by $W_n\omega_X$. 
Similarly, for a sheaf $\mathcal{A}$ of $\sO_X$-algebras, we obtain a sheaf $W_n\mathcal{A}$ of $W_n\mathcal{O}_X$-algebras, equipped with the corresponding $W_n\mathcal{O}_X$-module homomorphisms $F$, $V$ and $R$ as above. 
If $\pi \colon Y \to X$ is a morphism of $\F_p$-schemes, then functoriality gives a morphism $W_n \pi:W_n Y \to W_nX$ for each $n \ge 1$. 
If $\mathcal A$ is a sheaf of $\mathcal O_Y$-algebras, we often write $\pi_*W_n\mathcal{A}$ for $(W_n\pi)_*W_n\mathcal{A}$. 

For the rest of this section, suppose that $(A,\m)$ is a Noetherian local ring. 
Then $W_nA$ is a local ring of characteristic $p^n$ with maximal ideal
\[
\m_n=\{(a_0, \dots, a_{n-1}) \in W_nA \mid a_0 \in \m\}
\]
and the residue field $W_nA/\m_n$ is isomorphic to $A/\m$. 
The local cohomology module $H^i_{\m_n}(M)$ of a $W_nA$-module $M$ is often denoted simply by $H^i_\m(M)$. 
This notation is justified because $\Spec W_nA$ is naturally identified with $\Spec A$ as a topological space and $\m_n$ corresponds to $\m$ under this identification (see, for example, \cite{STY}*{\S 2.8} for further details). 

\begin{proposition}\label{Witt CM}
If $(A,\m)$ is a Cohen--Macaulay $F$-finite local ring of characteristic $p>0$, then $W_nA$ is also Cohen--Macaulay for all integers $n \ge 1$. 
\end{proposition}
\begin{proof}
For every integer $n \ge 1$, we have the exact sequence 
\[
0 \to F_*W_nA \xrightarrow{V} W_{n+1}A \xrightarrow{R^{n}} A \to 0.
\]
The Cohen--Macaulayness of a local ring $(R, \m_R)$ is equivalent to the vanishing of $H^i_{\m_R}(R)$ for all $i<\dim R$, so an inductive argument shows that $H^i_\m(W_nA)=0$ for all $i<\dim A=\dim W_n A$. 
\end{proof}

Let $\pi \colon X \to \Spec A$ be a morphism and $\mathcal{F}$ be a $W_n\sO_X$-module. 
The $W_n A$-module $H^i_{\m}(\mathcal{F})$ is defined as  
\[
H^i_{\m}(\mathcal{F}):=H^i_{\pi^{-1}(\{\m\})}(X, \mathcal{F}) \cong H^i(R\Gamma_{\m_n}R\pi_*\mathcal{F}),
\]
where $\m_n$ is the maximal ideal of $W_nA$. 
Let $E_n$ be an injective hull of the residue field $W_nA/\m_n\cong A/\m$ as a $W_nA$-module.
For a $W_nA$-module $M$, its Matlis dual is denoted by
\[
M^\vee := \Hom_{W_nA}(M,E_n).
\]
The $\m_n$-adic completion of $M$ is denoted by $M^{\wedge}$. 
We will use the following form of Matlis duality. 

\begin{lemma}[\textup{\cite{TWY}*{Lemma~2.9}, \cite{STY}*{Lemma~2.28}}]\label{lem:local duality}
Suppose that $X$ is a $d$-dimensional scheme and that $\pi \colon X \to \Spec A$ is a proper morphism. 
If $\mathcal{F}$ is a coherent $W_n\sO_X$-module, then there exists a natural isomorphism
\[
H^{d}_{\m}(\mathcal{F})^{\vee} \simeq \Hom_{W_n \sO_X}(\mathcal{F},W_n \omega_X)^{\wedge}.
\]
\end{lemma}

%%%%%%%%%%%%%%%%%%%%%%%%%%%%%%%%%%%%%%%%%%%%%%%%%%%%%%%%%%%%%%%%%

\section{\texorpdfstring{$F$}{F}-singularities}
In this section, we recall the definition of $F$-singularities and give a brief overview of their relationship with singularities arising in birational geometry.
The reader is referred to \cite{TW18} for further details. 

\begin{definition}[\cite{FW89}]\label{F-inj def}
Let $(R,\m)$ be a $d$-dimensional excellent local ring of characteristic $p>0$. 
\begin{enumerate}
\item[(i)] 
We say that $R$ is \textit{$F$-injective} if for all integers $i$, the map 
\[
H^i_{\m}(R) \to H^i_{\m}(F_*R) 
\]
induced by the Frobenius map $F \colon R \to F_*R$ is injective. 
We use the same letter $F$ to denote this map. 
\item[(ii)] 
We say that $R$ is \textit{$F$-rational} if $R$ is Cohen--Macaulay and if for every element $c \in R^{\circ}$, there exists an integer $e \ge 1$ such that the composite map 
\[
H^d_{\m}(R) \xrightarrow{F^e} H^d_{\m}(F^e_*R) \xrightarrow{\cdot F^e_*c} H^d_{\m}(F^e_*R)
\]
of the $e$-times iteration $F^e$ of the map $F\colon H^d_{\m}(R) \to H^d_{\m}(F_*R)$ and the endomorphism on $H^d_{\m}(R)$ defined by multiplication by $F^e_*c$ is injective. 
\end{enumerate}
\end{definition}

In commutative algebra, a local property $P$ is said to \textit{deform} if for a nonzerodivisor $x$ on $R$, the local ring $R$ satisfies property $P$ whenever the quotient ring $R/(x)$ does. 
\begin{remark}
$F$-rationality deforms, and $F$-injectivity deforms under the additional assumption that the ring is Cohen--Macaulay (see \cite{Fedder83} and \cite{HH94}). 
It is a long-standing open problem whether $F$-injectivity deforms without the Cohen--Macaulay assumption.
\end{remark}

The above singularities are related to Hodge theoretic singularities. 
\begin{theorem}[\cite{Schwede09_2}, \cite{Smith97}]\label{Finjective=>DB}
With the notation as in Definition~\ref{F-inj def}, suppose in addition that $R$ has a canonical module. 
If $R$ is $F$-injective $($resp.~$F$-rational$)$, then it is pseudo-Du~Bois $($resp.~pseudo-rational$)$.     
\end{theorem}

\begin{definition}[\cite{HR76}, \cite{HH89}, \cite{hw02}]\label{F-pure def}
Let $(R,\m)$ be an $F$-finite normal local ring of characteristic $p>0$ and $\Delta$ be an effective $\Q$-Weil divisor on $X=\Spec R$. 
\begin{enumerate}
\item[(i)] We say that the pair $(R,\Delta)$ is \textit{$F$-pure} if for all integers $e \ge 1$, the composite map 
\[
R \xrightarrow{F^e} F^e_*R \to F^e_*R((p^e-1)\Delta)
\]
of the $e$-times iterated Frobenius map $F^e$ and the pushforward via $F^e$ of the natural inclusion $R \hookrightarrow R((p^e-1)\Delta)$ splits as an $R$-module homomorphism. 
\item[(ii)] We say that the pair $(R,\Delta)$ is \textit{strongly $F$-regular} if for all elements $c \in R^{\circ}$, there exists an integer $e \ge 1$ such that the composite map 
\[
R \xrightarrow{F^e} F^e_*R \to  F^e_*R((p^e-1)\Delta) \xrightarrow{\cdot F^e_*c} F^e_*R((p^e-1)\Delta)
\]
splits as an $R$-module homomorphism, where the last map is multiplication by $F^e_*c$. 
\end{enumerate}
When $\Delta=0$, we simply say that $R$ is $F$-pure or strongly $F$-regular.  
\end{definition}

\begin{remark}\label{F-sing basic}
Let the notation be as in Definition~\ref{F-pure def}. 
\begin{enumerate}
\item 
If $R$ is strongly $F$-regular, then it is $F$-pure and $F$-rational. 
If $R$ is either $F$-pure or $F$-rational, then $R$ is $F$-injective. 
\item 
(cf.~\cite{KTTWYY1}*{Lemma~2.19}) If $\Delta$ is an effective $\Z_{(p)}$-Weil divisor, then $(R,\Delta)$ is $F$-pure if and only if the map $R \to F^e_*R(p^e \Delta -\lfloor \Delta \rfloor)$ splits for all integers $e \ge 1$.  
\item 
(cf.~\cite{TW18}*{Lemma~3.12}) By Matlis duality, the pair $(R,\Delta)$ is $F$-pure if and only if the map 
\[
H^d_{\m}(\omega_R) \to H^d_{\m}(F^e_*R(p^eK_R+(p^e-1)\Delta)), 
\]
induced by tensoring the map $R \to F^e_*R((p^e-1)\Delta)$ with $H^d_{\m}(\omega_R)$, is injective for all integers $e \ge 1$. 
An analogous description holds for strong $F$-regularity. 
In particular, if $R$ is quasi-Gorenstein, that is, $\omega_R \cong R$, then $R$ is $F$-injective (resp.~$F$-rational) if and only if $R$ is $F$-pure (resp.~strongly $F$-regular).
\item 
(cf.~\cite{TT08}*{Lemma~4.5}) $(R,\Delta)$ is strongly $F$-regular if and only if for every effective Cartier divisor $D$ on $\Spec R$, there exists a rational number $\varepsilon>0$ such that $(R, \Delta+ \varepsilon D)$ is $F$-pure. 
\item 
Strong $F$-regularity and $F$-purity do not deform in general (see \cite{Singh99b} and \cite{Fedder83}). However, they deform for $\Q$-Gorenstein normal rings (see \cite{AKM98}, \cite{hw02} and \cite{PS23}). 
A stronger form of these deformation results is given by inversion of adjunction type results for strong $F$-regularity and $F$-purity. 
The reader is referred to \cite{Das15} and \cite{schwede09} for further details.
\end{enumerate}
\end{remark}

The following criterion for $F$-purity and strong $F$-regularity is known as a Fedder-type criterion and is particularly useful for hypersurface singularities. 
\begin{proposition}\label{Fedder-Fpure} 
Let $S=k[[X_1, \dots, X_n]]$ be a formal power series ring over an $F$-finite field of characteristic $p>0$ and $I$ be a nonzero ideal of $S$. 
Set $R:=S/I$. 
\begin{enumerate}
\item[(1)] $($\cite{Fedder83}$)$ $R$ is $F$-pure if and only if $(I^{[p]}:I) \not\subseteq \m^{[p]}$. 
In particular, when $I$ is a principal ideal generated by $f \in S$, the hypersurface $R$ is $F$-pure if and only if $f^{p-1} \notin \m^{[p]}$. 
\item[(2)] $($\cite{Glassbrenner96}$)$ Fix an element $c \in S \setminus I$ such that the localization $R_c=S_c/IS_c$ with respect to $c$ is a regular ring. 
Then $R$ is strongly $F$-regular if and only if there exists an integer $e \ge 1$ such that $c(I^{[p^e]}:I) \not\subseteq \m^{[p^e]}$. 
In particular, when $I$ is a principal ideal generated by $f \in S$, the hypersurface $R$ is strongly $F$-regular if and only if there exists an integer $e \ge 1$ such that $cf^{p^e-1} \notin \m^{[p^e]}$.  
\end{enumerate}
\end{proposition}

The following theorem relates $F$-purity and strong $F$-regularity to singularities of the minimal model program. 
\begin{theorem}[\cite{hw02}]\label{Fpure=>lc}
Let the notation be as in Definition~\ref{F-pure def}. 
If $(R,\Delta)$ is $F$-pure $($resp.~strongly $F$-regular$)$, then $(X=\Spec R, \Delta)$ is lc $($resp.~klt$)$. 
\end{theorem}

Neither the converse of Theorem~\ref{Finjective=>DB} nor that of Theorem~\ref{Fpure=>lc} holds even in dimension two. 
\begin{example}\label{example:Fsing}
Let $k[[X,Y,Z]]$ denote the three-dimensional formal power series ring over a perfect field $k$ of characteristic $p>0$. 

\textup{(1)}
Let $R=k[[X,Y,Z]]/(X^2+Y^3+Z^5)$. Then $\Spec R$ has klt singularities, but Proposition~\ref{Fedder-Fpure} (2) shows that $R$ is strongly $F$-regular (equivalently, $F$-rational) if and only if $p>5$. 

More generally, suppose that $R$ is a two-dimensional $F$-finite normal local ring of characteristic $p$ and $\Delta$ is an effective $\Q$-Weil divisor on $X=\Spec R$ whose coefficients belong to the standard set $\{1-1/m \mid m \in \Z_{>0}\}$. 
If $p>5$, then it is known by \cite{hara98} and \cite{satotakagi} that  $(X,\Delta)$ is klt if and only if $(R,\Delta)$ is strongly $F$-regular. 

\medskip

\textup{(2)}
Let $R=k[[X,Y,Z]]/(X^3+Y^3+Z^3)$. Then $\Spec R$ has lc singularities, but Proposition~\ref{Fedder-Fpure} (2) shows that $R$ is $F$-pure (equivalently, $F$-injective) if and only if $p \equiv 1 \pmod 3$. More generally, if $R$ is the local ring at the vertex of the affine cone over an elliptic curve $E$ over $k$, then $\Spec R$ has lc singularities, and $R$ is $F$-pure if and only if $E$ is ordinary. 
\end{example}

From the above examples, we see that even in dimension two, $F$-purity and $F$-injectivity are too restrictive compared to log canonicity and being Du Bois, respectively.  
This is one of the motivations for considering quasi-$F$-singularities.  

\begin{definition}[\cite{tw04}]
Let $(R,\m)$ be an $F$-finite normal local ring of characteristic $p>0$ and $\Delta$ be an effective $\Q$-Weil divisor on $X:=\Spec R$. 
Then the $F$-pure threshold $\mathrm{fpt}(R;\Delta)$ of $\Delta$ is defined as 
\[
\fpt(R;\Delta)=\sup\{t \in \Q_{\ge 0} \mid (R,t \Delta) \textup{ is $F$-pure}\}.
\]
If $\Delta$ is a Cartier divisor $\operatorname{div}(f)$, then $\fpt(R;\Delta)$ is also denoted by $\fpt(R;f)$. 
\end{definition}

\begin{remark}
When $X$ is $\Q$-Gorenstein and $\Delta$ is $\Q$-Cartier, Theorem~\ref{Fpure=>lc} says that $\mathrm{fpt}(R;\Delta) \le \mathrm{lct}(X;\Delta)$. 
\end{remark}

When $R$ is regular, the $F$-pure thresholds can be described in a way similar to the Fedder-type criterion (Proposition~\ref{Fedder-Fpure}). 

\begin{proposition}[\cite{MTW}]\label{Fedder-fpt}
Let $R=k[[X_1, \dots, X_n]]$ be the $n$-dimensional formal power series ring over an $F$-finite field of characteristic $p>0$ and $f \in (X_1, \dots, X_n)$ be a nonzero element. 
For every integer $e \ge 1$, set
\[
\nu_e(f):=\max\{r \in \Z_{\ge 0} \mid f^r \notin (X_1^{p^e}, \dots, X_n^{p^e})\}.
\]
Then the sequence $\{\nu_e(f)/p^e\}_{e \ge 1}$ is monotonically increasing and one has 
\[
\fpt(R;f)=\lim_{e \to \infty} \frac{\nu_e(f)}{p^e}=\sup_e \frac{\nu_e(f)}{p^e}. 
\]
\end{proposition}

Using the above description, we can compute the $F$-pure threshold $\fpt(R;f)$ for a relatively simple $f$. 

\begin{example}[\cite{MTW}*{Example~4.3}]\label{fpt example}
Let $R=k[[X,Y]]$ be the two-dimensional formal power series ring over a perfect field $k$ of characteristic $p>0$ and $f=X^3+Y^2 \in R$. 
Then 
\[
\fpt(R;f) =
\begin{cases}
  1/2 & (p = 2), \\[2mm]
  2/3 & (p = 3), \\[2mm]
  (5p - 1)/(6p) & (p \equiv 5 \;(\mathrm{mod}\ 6)), \\[2mm]
  5/6 & (p \equiv 1 \;(\mathrm{mod}\ 6)). 
\end{cases}
\]
\end{example}

%%%%%%%%%%%%%%%%%%%%%%%%%%%%%%%%%%%%%%%%%%%%%%%%%%%%%%%%%%%%%%

\section{Quasi-\texorpdfstring{$F$}{F}-singularities}
In this section, we introduce quasi-$F$-singularities as a generalization of $F$-singularities and discuss their basic definitions and properties. 

\subsection{Definitions}

First, we recall the definition of quasi-$F$-splitting introduced by Yobuko, which is the origin of quasi-$F$-singularities. 
\begin{definition}[\cite{yobuko19}]
For an integer $n \ge 1$, an $F$-finite local domain $A$ of characteristic $p>0$ is said to be \emph{$n$-quasi-$F$-split} if there exists a $W_nA$-module homomorphism $\varphi \colon F_*W_nA \to A$ such that $\varphi \circ F=R^{n-1}$. 
We say that $A$ is \textit{quasi-$F$-split} if $A$ is $n$-quasi-$F$-split for some $n \ge 1$. 

Let $Q_{A,n}$ be the pushout of $F \colon W_nA \to F_*W_nA$ and $R^{n-1}:W_nA \to A$. 
By \cite{TWY}*{Proposition~3.12}, $Q_{A,n}$ is naturally an $A$-module, and hence the map $\Phi_{A,n}: A \to Q_{A,n}$ may be viewed as a map of $A$-modules.
By definition, $A$ is $n$-quasi-$F$-split if and only if the map $\Phi_{A,n}: A \to Q_{A,n}$ splits as an $A$-module homomorphism. 
\[
\xymatrix{
W_n A \ar[r]^F \ar[d]_{R^{n-1}} & F_*W_n A \ar@{.>}[dl]_{\varphi} \ar[d] \\
A \ar[r]^{\Phi_{A,n}} & Q_{A,n}
}
\]
\end{definition}

\begin{definition}\label{def:Q^{e,S}_{X,D,n}}
Let $X$ be an $F$-finite normal scheme of characteristic $p>0$ with function field $\mathscr{K}_X$, and let $D$ be a $\Q$-Weil divisor on $X$. 
\begin{enumerate}
\item[(i)] \textup{(\cite{tanaka22})}
We define the $W_n\mathcal{O}_X$-submodule $W_n\mathcal{O}_X(D)$ of the constant sheaf $W_n\mathscr{K}_X$ as follows. 
For each open subset $U \subseteq X$, we set 
\begin{align*}
\hspace*{3em} \Gamma(U, W_n\mathcal{O}_X(D))
&\coloneqq 
\bigl\{\,(\varphi_0,\ldots,\varphi_{n-1}) 
 \bigm| 
\varphi_i \in \Gamma(U,\mathcal{O}_X(p^iD))
\text{ for all } i \,\bigr\} \\
&\subseteq \Gamma(U, W_n\mathscr{K}_X).
\end{align*}

\item[(ii)]
We define the $W_n\mathcal{O}_X$-module $Q^{e}_{X,D,n}$ and the morphism
\[
\Phi^{e}_{X,D,n} \colon \mathcal{O}_X(D) \longrightarrow Q^{e}_{X,D,n}
\]
by the pushout diagram
\begin{equation}\label{eq:def-Q^{e,S}_{X,D,n}}
\begin{tikzcd}
W_n\mathcal{O}_X(D) \arrow[r,"F^e"] \arrow[d,"R^{n-1}"'] &
F_*^e\bigl(W_n\mathcal{O}_X(p^eD)\bigr) \arrow[d] \\
\mathcal{O}_X(D) \arrow[r,"\Phi^{e}_{X,D,n}"] &
Q^{e}_{X,D,n}.
\end{tikzcd}
\end{equation}
If $D=0$, we write $Q^e_{X,n}$; if $e=1$, we write $Q_{X,D,n}$; and if both $D=0$ and $e=1$, we simply write $Q_{X,n}$. 
When $X=\Spec A$ is affine, with $(A,\m)$ a $d$-dimensional $F$-finite local ring, we write $Q^{e}_{A,D,n}$ in place of $Q^{e}_{X,D,n}$, following the same abbreviation conventions.
The same conventions apply to $\Phi^{e}_{X,D,n}$. 
Moreover, the map induced by applying $H^d_{\m}(-)$ to $\Phi^{e}_{A,D,n}$ is denoted by the same symbol.

\item[(iii)] Assume that $X=\Spec A$ is affine, with $(A,\m)$ a $d$-dimensional $F$-finite local ring. 
For each $c \in A^{\circ}$, we define the $W_nA$-module $Q^e_{A, D, n}(c)$ and the morphism 
\[
\Phi^e_{A,D, n}(c) \colon A(D) \to Q^e_{A, D, n}(c)
\]
by the pushout diagram 
\begin{equation}\label{eq:def-Q^{e}_{A,D,n}(c)}
\begin{tikzcd}
W_nA(D) \arrow[r,"F^e"] \arrow[d,"R^{n-1}"'] &
F_*^e\bigl(W_nA(p^eD)\bigr) \arrow[r,"\cdot F^e_*{[c]}"] &  F_*^e\bigl(W_nA(p^eD)\bigr) \arrow[d]  \\
A(D) \arrow[rr,"\Phi^{e}_{A,D,n}(c)"] & & 
Q^{e}_{A,D,n}(c). 
\end{tikzcd}
\end{equation}
Here $F_*^e\bigl(W_nA(p^eD)\bigr) \xrightarrow{\cdot F^e_*{[c]}} F_*^e\bigl(W_nA(p^eD)\bigr)$ is multiplication by $F^e_*{[c]}$. 
We use abbreviation conventions similar to those in (ii). 
\end{enumerate}
\end{definition}

\begin{remark}\label{Q remark}
Although $Q_{A,n}$ is naturally an $A$-module, $Q^e_{A,n}$ no longer has an $A$-module structure for $e>1$.
\end{remark}

\begin{lemma}\label{quasi-F-split equivalent def}
A $d$-dimensional $F$-finite normal local ring $A$ of characteristic $p>0$ is quasi-$F$-split if and only if, for some integer $n \ge 1$, the map 
\[
\Phi_{A,K_A, n}: H^d_{\m}(\omega_A) \to H^d_{\m}(Q_{A,K_A,n})\]
is injective. 
\end{lemma}
\begin{proof}
By definition, $A$ is quasi-$F$-split if and only if for some integer $n \ge 1$, the map $\Hom_A(Q_{A,n}, A) \to \Hom_A(A,A)=A$ induced by $\Phi_{A,n}:A \to Q_{A,n}$ is surjective. 
By Matlis duality, this is equivalent to saying that 
\[\Phi_{A,n} \otimes_A \omega_A: H^d_{\m}(\omega_A) \to H^d_{\m}(Q_{A,n} \otimes_A \omega_A)\] is injective. 

We now compare this map with $\Phi_{A,K_A,n}$. 
Applying the right exact functors $-\otimes_{W_nA} W_nA(K_A)$ and $H^d_{\m}(-)$ to the defining pushout diagram of $Q_{A,n}$ gives a pushout diagram. 
In the upper-right term of this pushout diagram, we identify
\[
H^d_{\m}(F_*W_nA\otimes_{W_nA}W_nA(K_A)) \cong H^d_{\m}(F_*W_nA(pK_A))
\]
using the following observation applied to the natural map
\[
F_*W_nA\otimes_{W_nA}W_nA(K_A)\to F_*W_nA(pK_A). 
\]
If a homomorphism $M\to N$ of finitely generated $A$-modules is an isomorphism in codimension one, then it induces an isomorphism $H^d_{\m}(M)\cong H^d_{\m}(N)$. 
Thus we obtain a pushout diagram
\[
\xymatrix{
H^d_{\m}(W_nA(K_A)) \ar[r]^{\hspace*{-1em} F} \ar[d]_{R^{n-1}}
& H^d_{\m}(F_*W_nA(pK_A)) \ar[d]  \\
H^d_{\m}(\omega_A) \ar[r]^{\hspace*{-2em}\Phi_{A,n} \otimes_A \omega_A}
& H^d_{\m}(Q_{A,n} \otimes_A \omega_A).
}
\]

On the other hand, by the definition of $Q_{A,K_A,n}$, applying $H^d_{\m}(-)$ to its defining pushout diagram gives another pushout diagram with the same three remaining terms and the same structure maps, but with lower-right term $H^d_{\m}(Q_{A,K_A,n})$. 
Hence, by the uniqueness of pushouts, we obtain a canonical isomorphism
\[
H^d_{\m}(Q_{A,n}\otimes_A\omega_A) \cong H^d_{\m}(Q_{A,K_A,n}),
\]
under which $\Phi_{A,n}\otimes_A\omega_A$ is identified with
$\Phi_{A,K_A,n}$. 
This completes the proof.
\end{proof}

\begin{setting}\label{notation:qFinj}
Let $(A,\m)$ be a $d$-dimensional $F$-finite normal local ring of characteristic $p>0$. 
Let $\Delta=\sum_i a_i D_i$ be a boundary $\Q$-Weil divisor on $X:=\Spec A$, that is, a $\Q$-Weil divisor on $X$ with $0 \le a_i \le 1$ for all $i$. 
For all integers $e, n \ge 1$, we set
\begin{align*} 
K^{e}_{n} &:= \Ker\bigl( H^d_\m(W_nA) \xrightarrow{F^e} F^e_*H^d_\m(W_nA) \bigr), \\ 
K^{e}_{\infty} &:= \Ker\bigl( \varprojlim_n H^d_\m(W_nA) \xrightarrow{F^e} \varprojlim_n F^e_*H^d_\m(W_nA) \bigr). 
\end{align*}
By the left exactness of inverse limits, we have $K^e_\infty=\varprojlim_n K^e_n$. 
\end{setting}

Motivated by Lemma~\ref{quasi-F-split equivalent def}, one defines quasi-$F$-singularities by replacing the $e$-times iterated Frobenius maps 
\[
A \to F^e_*A \textup{ and } H^d_{\m}(\omega_A) \to H^d_{\m}(F^e_*A(p^eK_A))
\]in the definition of $F$-singularities with the maps 
\[\Phi^e_{A,n}: A \to Q^e_{A,n} \textup{ and } \Phi^e_{A,K_A,n}: H^d_{\m}(\omega_A) \to H^d_{\m}(Q^e_{A,K_A,n}).\] 
Compare the following definition with Definition~\ref{F-inj def} and Definition~\ref{F-pure def}, together with Remark~\ref{F-sing basic}. 
\begin{definition}\label{def:quasi-F-invariants}
With the notation as in Setting~\ref{notation:qFinj}, let $e \ge 1$ be an integer.
\begin{enumerate}
\item[\textup{(i)}] (cf.~\cite{Kaw7})
We say that $A$ is \emph{quasi-$F^e$-injective} if for some integer $n \ge 1$, the map
\[
\Phi^e_{A,n} \colon H^i_\m(A) \longrightarrow H^i_\m(Q^e_{A,n})
\]
is injective for every integer $i$.
When $e = 1$, we simply say that $A$ is \emph{quasi-$F$-injective}.

\item[\textup{(ii)}] (\cite{TWY})
When $\lfloor \Delta \rfloor=0$, we say that $(A,\Delta)$ is \emph{quasi-$F^e$-split} if for some $n \ge 1$, the map
\[
\Phi^{e}_{A,K_A+\Delta,n} \colon 
H^d_\m(\omega_A) \longrightarrow H^d_\m(Q^{e}_{A,K_A+\Delta,n})
\]
is injective. 
This definition can be extended to the case where $\lfloor \Delta \rfloor \ne 0$.
In that setting, the pair is said to be \textit{purely quasi-$F^e$-split}. 
We refer the reader to \cite{TWY}*{Definition 3.34} for the precise definition. 

\item[\textup{(iii)}] (\cite{KTTWYY3})
We say that $A$ is \emph{quasi-$F$-rational} if $A$ is Cohen--Macaulay and if for every $c \in A^{\circ}$, 
there exist integers $e,n \ge 1$ such that the map 
\[
\Phi^e_{A,n}(c) \colon H^d_\m(A) \to H^d_\m(Q^e_{A,n}(c))
\]
is injective.

\item[\textup{(iv)}] (\cite{TWY})
We say that $(A,\Delta)$ is \emph{quasi-$F$-regular} 
if $\lfloor \Delta \rfloor=0$ and if for every $c \in A^{\circ}$, 
there exist integers $e, n \ge 1$ such that the map  
\[
\Phi^e_{A,K_A+\Delta,n}(c) \colon H^d_\m(\omega_A) \to H^d_\m(Q^e_{A,K_A+\Delta,n}(c))
\]
is injective. 
\end{enumerate}
We say that $A$ is quasi-$F^{\infty}$-injective if it is quasi-$F^e$-injective for all $e\ge 1$. 
Similarly, $(A,\Delta)$ is said to be (purely) quasi-$F^{\infty}$-split if it is (purely) quasi-$F^e$-split for all $e\ge 1$. 
\end{definition}

\begin{remark}
By definition, quasi-$F^{\infty}$-splitting means that for every integer $e\ge 1$, there exists an integer $n\ge 1$ such that the ring is quasi-$F^e$-split.
One can also consider \textit{uniform quasi-$F^{\infty}$-splitting}, where the integer $n$ can be chosen independently of $e$.
By \cite{TWY}*{Corollary~7.3}, the affine cone over a supersingular elliptic curve is quasi-$F^{\infty}$-split but not uniformly quasi-$F^{\infty}$-split. 
Since the uniform version is too restrictive for our purposes, we do not pursue it further in this paper.
\end{remark}

\begin{remark}\label{remark:quasi-F-singularities}
Keeping the notation of Setting~\ref{notation:qFinj}, we record two basic properties of quasi-$F$-splitting and quasi-$F$-regularity. 
\begin{enumerate}
\item
For all $e \ge 1$, quasi-$F^{e+1}$-splitting implies quasi-$F^{e}$-splitting. 
In particular, quasi-$F^{\infty}$-splitting implies quasi-$F$-splitting. 
If $A$ is Gorenstein, or if $\dim A=2$, then $A$ is quasi-$F$-split if and only if $A$ is quasi-$F^{\infty}$-split by \cite{KTTWYY3}*{Theorem~H} and \cite{STY}*{Proposition~3.16}, respectively.  
However, Baudin \cite{Bau26} recently informed the authors that these two notions are not equivalent in general. 
\item
$(A,\Delta)$ is quasi-$F$-regular if and only if for every effective Cartier divisor $D$ on $\Spec A$, there exists an integer $e \ge 1$ such that $(A,\Delta + (1/p^e)D)$ is quasi-$F^e$-split. 
This can be viewed as a quasi-$F$ analog of the characterization of strong $F$-regularity in Remark~\ref{F-sing basic}(4).
\end{enumerate}
\end{remark}

\begin{proposition}\label{cor:qFinj-to-qF^einj}
\textup{(cf.~\cite{STY}*{Corollary~3.12}; see also \cite{KTTWYY3}*{Theorem~8.5})}
With the notation in Setting~\ref{notation:qFinj}, assume that $A$ is Cohen--Macaulay.
Then the following conditions are equivalent to each other. 
\begin{enumerate}
\item[\textup{(i)}] $A$ is quasi-$F$-injective. 
\item[\textup{(ii)}] $A$ is quasi-$F^{\infty}$-injective. 
\item[\textup{(iii)}] The map
\[
\varprojlim_{n} H^d_{\m}(W_nA)
\xrightarrow{F}
\varprojlim_{n} F_* H^d_{\m}(W_nA)
\]
is injective.
\item[\textup{(iv)}]  The map
\[
\varprojlim_{n} H^d_{\m}(W_nA)
\xrightarrow{F^e}
\varprojlim_{n} F^e_* H^d_{\m}(W_nA)
\]
is injective for all $e \ge 1$.
\end{enumerate}
\end{proposition}

\begin{proof}
The equivalence of \textup{(i)} and \textup{(ii)} follows from the proof of \cite{KTTWYY3}*{Theorem~8.5}. 
For later use, however, we give a proof that also shows the equivalence with \textup{(iii)} and \textup{(iv)}.

First, note that $W_nA$ is Cohen--Macaulay by Proposition~\ref{Witt CM}, and therefore $H^i_{\m}(W_nA)=0$ for all $i<d$. 
We begin with the following claim. 
\begin{claim}\label{prop:infty-qFinj}
\textup{(cf.~\cite{STY}*{Proposition~3.10})}
For an integer \(e \ge 1\), the ring $A$ is quasi-\(F^e\)-injective if and only if the natural projection
\[
K^{e}_{\infty} \longrightarrow K^{e}_{1}
\]
is the zero map. 
\end{claim}
\begin{proof}[Proof of Claim~\ref{prop:infty-qFinj}]
Since $H^d_{\m}(-)$ is a right exact functor, 
\[
\xymatrix{
H^d_{\m}(W_nA) \ar[r]^{F^e} \ar[d]_{R^{n-1}} &
F^e_*H^d_{\m}(W_nA) \ar[d]  \\
H^d_{\m}(A) \ar[r]^{\Phi^e_{A,n}} &
H^d_{\m}(Q^e_{A,n}) \\
}
\]
is a pushout diagram.
Noting that $R^{n-1}\colon H^d_{\m}(W_nA) \to H^d_{\m}(A)$ is surjective, we may identify 
$\Phi^e_{A,n}$ with the map 
\[
H^d_{\m}(A) \to
F^e_*H^d_{\m}(W_nA)/F^e(\Ker R^{n-1});
\quad
R^{n-1}(x) \mapsto \overline{F^e(x)}.
\]
Under this identification, $A$ is quasi-$F^e$-injective
if and only if for some integer $\ell \ge 1$,
the projection $K^{e}_{\ell} \to K^{e}_{1}$ is the zero map. 
Thus, it suffices to show that the map 
$\varprojlim_{n} K^{e}_{n} \to K^{e}_{1}$ 
is the zero map if and only if the truncation 
$K^{e}_{\ell} \to K^{e}_{1}$
is the zero map for some $\ell \ge 1$. 
The ``if'' part is clear, so it remains to prove the ``only if'' part. 
For each integer $n \ge 1$, set
\[
I_n := \bigcap_{m \ge n}
\Image\bigl( K^{e}_{m} \to K^{e}_{n} \bigr)
\subseteq K^{e}_{n}.
\]
There exists an integer $m_n \ge n$ such that 
$I_n = \Image\bigl( K^{e}_{m_n} \to K^{e}_{n} \bigr)$, because $K^{e}_{n}$ is Artinian.  
It follows that the transition map $I_{n+1} \to I_n$ is surjective for all $n \ge 1$, and hence the map 
$\varprojlim_n I_n \to I_1$
is surjective by the Mittag--Leffler condition for inverse systems (see \cite{stacks-project}*{Tag~0594}). 
If $K^{e}_{\infty} \to K^{e}_{1}$ is the zero map,
then so is $\varprojlim_n I_n \to I_1$, which yields $I_1=0$. 
This completes the proof.  
\end{proof}

The implications (iii) $\Leftrightarrow$ (iv) and (ii) $\Rightarrow$ (i) are clear, and (iii) $\Rightarrow$ (i) and (iv) $\Rightarrow$ (ii) 
follow from Claim~\ref{prop:infty-qFinj}. 
It remains to verify (i) $\Rightarrow$ (iii).
By (i) and the proof of Claim~\ref{prop:infty-qFinj}, there exists an integer $m \ge 1$ such that the map
$K^{1}_m \to K^1_1$ is zero.
We show by induction on $n$ that the map
$K^1_{m+n} \to K^1_{n+1}$
is zero for every integer $n \ge 0$. 

The case $n=0$ is clear from the choice of $m$. Assume that $n \ge 1$. 
Since $A$ is Cohen--Macaulay, we have the following commutative diagram with exact rows:
\[
\xymatrix{
0 \ar[r] & F^2_*H^d_{\m}(W_{m+n-1}A) \ar[r]^{F_*V} & F_*H^d_{\m}(W_{m+n}A) \ar[r]^{\hspace*{2.5em}F_*R^{m+n-1}} & F_*H^d_{\m}(A) \ar[r] & 0 \\
0 \ar[r] & F_*H^d_{\m}(W_{m+n-1}A) \ar[u]_{F_*F} \ar[r]^V \ar[d]^{F_*R^{m-1}} & H^d_{\m}(W_{m+n}A) \ar[u]_F \ar[r]^{\hspace*{2em}R^{m+n-1}} \ar[d]^{R^{m-1}} & H^d_{\m}(A) \ar[u]_F \ar[r] \ar@{=}[d] & 0 \\
0 \ar[r] & F_*H^d_{\m}(W_{n}A) \ar[r]^V & H^d_{\m}(W_{n+1}A) \ar[r]^{R^{n}} & H^d_{\m}(A) \ar[r] &  0 \\
}
\]
Let $\xi \in K^1_{m+n}$. By the choice of $m$, the composition $K^1_{m+n} \to K^1_m \to K^1_1$ is zero, so $R^{m+n-1}(\xi)=0$. 
By the exactness of the middle row, there exists an element $\eta \in H^d_{\m}(W_{m+n-1}A)$ such that $\xi=V(F_*\eta)$. 
Since $\xi \in K^1_{m+n}$, the commutativity of the above diagram gives 
\[
F_*V(F_*(F(\eta)))=F(\xi)=0.  
\]
It follows from the injectivity of $F_*V$ that $F(\eta)=0$, that is, $\eta\in K^1_{m+n-1}$.
By the induction hypothesis, $R^{m-1}(\eta)=0$, and therefore  
\[R^{m-1}(\xi)=R^{m-1}(V(F_*\eta))=V\bigl(F_*(R^{m-1}(\eta))\bigr)=0.\]
Thus the map $K^1_{m+n} \to K^1_{n+1}$ is zero.
Passing to the inverse limit, we obtain an isomorphism 
\[
K^1_\infty
= \varprojlim_m K^1_{n+m}
\xrightarrow{\simeq}
\varprojlim_m K^1_{m+1}
= K^1_\infty,
\]
which is zero. 
Hence, $K^1_\infty = 0$, that is, (iii) holds. 
\end{proof}

\subsection{Deformations of quasi-$F$-singularities}
The deformation behavior of quasi-$F$-singularities is more subtle than that of $F$-singularities.
As we will see in Example~\ref{example:fedder-qFs} and Remark~\ref{example-inv-adj}, quasi-$F$-splitting, quasi-$F$-regularity, quasi-$F$-rationality, and quasi-$F$-injectivity do not deform in general, even for Gorenstein rings. 
By contrast, we prove in this paper that quasi-$F$-rationality deforms for quasi-$F$-injective singularities.

\begin{remark}
There are also positive results for global deformation and inversion of adjunction under certain global assumptions (see \cite{KTTWYY1}*{Theorem~4.6 and Corollary~4.12}). 
These results include the cases of plt blow-ups and Koll\'ar components, which are used in the proof of Theorem~\ref{thm:quasi-F-split of 3-dim klt}. 
\end{remark}

First, we recall the notation used in the theory of quasi-$F$-rationality developed in \cite{KTTWYY3}.

\begin{definition}[cf.~\cite{KTTWYY3}*{Section~3}]\label{def:test-submod}
Let $(A,\m)$ be a $d$-dimensional $F$-finite normal local ring of characteristic $p>0$.
Fix integers $n \ge 1$ and $e\ge 0$.
\begin{enumerate}
\item[(i)]
For each $c \in A^{\circ}$, set 
\[
\widetilde{K^{e,c}_n}
:=
\Ker\left(
H^d_\m(W_nA)
\xrightarrow{F^e}
H^d_\m(F^e_*W_nA)
\xrightarrow{\cdot F^e_*[c]}
H^d_\m(F^e_*W_nA)
\right),
\]
which is a $W_nA$-submodule of $H^d_\m(W_nA)$. 
We then define $\widetilde{0_{A,n}^{*}}$ to be the $W_n A$-submodule of $H^d_\m(W_nA)$ consisting of all elements $z \in H^d_\m(W_nA)$ for which there exists $c \in A^{\circ}$ such that $z$ lies in $\widetilde{K^{e,c}_n}$ for all large $e$. 

\item[(ii)]
For each $c \in A^{\circ}$, set
\[
K^{e,c}_n:=\Ker\left(\Phi^e_{A,n}(c) \colon H^d_\m(A) \to H^d_\m(Q^e_{A,n}(c))\right),
\]
which is an $A$-submodule of $H^d_\m(A)$. 
The {\em $n$-quasi-tight closure} $0_{A,n}^{*}$ is then defined to be the $A$-submodule of $H^d_\m(A)$ consisting of all elements $z \in H^d_\m(A)$ for which there exists $c \in A^{\circ}$ such that $z$ lies in $K^{e,c}_n$ for all large $e$. Note by \cite{KTTWYY3}*{Theorem~3.25} that $R^{n-1}(\widetilde{0_{A, n}^{*}})=0_{A, n}^{*}$. 

\item[(iii)]
The {\em quasi-tight closure} $0_A^{q*}$ is the $A$-submodule of $H^d_\m(A)$ defined by 
\[
0_A^{q*}:=\bigcap_{n=1}^{\infty} 0_{A, n}^{*}. 
\]
The restriction maps $W_{n+1}A\to W_nA$ induces, by the universal property of pushouts, maps $Q^e_{A,n+1}(c)\to Q^e_{A,n}(c)$ 
compatible with the maps from $A$. 
Hence $K^{e,c}_{n+1}\subseteq K^{e,c}_n$, and so $0^*_{A,n+1}\subseteq 0^*_{A,n}$. 
Since $H^d_\m(A)$ is Artinian, this descending chain stabilizes. Thus, there exists an integer $m \ge 1$ such that $0_{A, m}^{*}=0_A^{q*}$. 
Note that $A$ is quasi-$F$-rational if and only if $0_A^{q*}=0$ and $A$ is Cohen--Macaulay. 

\item[(iv)]
We define submodules of $\omega_A$ by taking annihilators with respect to the duality pairing $\omega_A \times H^d_{\m}(A) \to H^d_{\m}(\omega_A)$: 
\[
\tau_n(\omega_A)
:=
\operatorname{Ann}_{\omega_A}(0_{A, n}^{*}),
\quad
\tau^q(\omega_A)
:=
\operatorname{Ann}_{\omega_A}(0_A^{q*}). 
\]
Then $\{\tau_n(\omega_A)\}_{n \ge 1}$ forms an ascending chain of submodules of $\omega_A$, and therefore there exists an integer $m \ge 1$ such that
$\tau_m(\omega_A)=\tau^q(\omega_A)$.
Note that $A$ is quasi-$F$-rational if and only if $A$ is Cohen--Macaulay and $\tau^q(\omega_A)=\omega_A$.
Furthermore, by \cite{KTTWYY3}*{Theorem~3.10},
the formation of $\tau_n(\omega_A)$ and $\tau^q(\omega_A)$ commutes with localization and completion.
\end{enumerate}
\end{definition}

\begin{theorem}\label{inv-adj}
With the notation as in Setting~\ref{notation:qFinj}, let $f \in A$ be a nonzero element. 
If $A$ is quasi-$F$-injective and $A/(f)$ is quasi-$F$-rational, then $A$ is quasi-$F$-rational.    
\end{theorem}

\begin{proof}
First, note that $A$ is Cohen--Macaulay, as is $A/(f)$. 
By the quasi-$F$-rationality of $A/(f)$, there exists an integer $h \ge 1$ such that $\tau_h(\omega_{A/(f)})=\omega_{A/(f)}$, equivalently $0^*_{A/(f),h}=0$. 
Increasing $h$ if necessary, we may also assume that $\tau_h(\omega_A)=\tau^q(\omega_{A})$. 
We will show that $\tau_h(\omega_{A})=\omega_A$. 
Since the formation of $\tau_h(\omega_A)$ commutes with localization, by induction on the dimension of $A$, we may assume that $\omega_A/\tau_h(\omega_A)$ has finite length. 

Let $A^{\circ, f}$ denote the set of elements of $A$ not in any minimal prime ideal of the principal ideal $(f)$. 
We then define $\wt{0^*_{A,f,h}}$ to be the $W_hA$-submodule of $H^{d-1}_\m(W_hA/([f]))$ consisting of all elements $z \in H^{d-1}_\m(W_hA/([f]))$ for which there exists $c \in A^{\circ, f}$ such that $F^e_*[c] F^e(z)=0$ for all large $e$. 
Let 
\[R^{h-1}_f \colon H^{d-1}_\m(W_hA/([f])) \to H^{d-1}_\m(A/(f))\] 
be the natural map induced by the restriction map. 

\begin{claim}\label{cl:inv1}
$R^{h-1}_f(\wt{0^*_{A,f,h}}) = 0$. 
\end{claim}
\begin{proof}[Proof of Claim~\ref{cl:inv1}]
The canonical surjection $A \to A/(f)$ induces a ring homomorphism $W_hA \to W_h(A/(f))$ that sends the ideal $([f])$ to zero. Therefore, we have a natural map 
\[
\iota_h \colon H^{d-1}_\m(W_hA/([f])) \longrightarrow H^{d-1}_\m(W_h(A/(f))). 
\]
By definition, $\iota_h(\wt{0^*_{A,f,h}}) \subseteq \wt{0^*_{A/(f),h}}$.
Noting that 
\[
R^{h-1}_f
= \bigl(
H^{d-1}_\m(W_h(A)/([f]))
\xrightarrow{\iota_h}
H^{d-1}_\m(W_h(A/(f)))
\xrightarrow{R^{h-1}}
H^{d-1}_\m(A/(f))
\bigr),
\]
we have 
\[
R^{h-1}_f(\wt{0^*_{A,f,h}}) \subseteq R^{h-1}(\wt{0^*_{A/(f),h}})=0^*_{A/(f),h}=0. 
\] 
\end{proof}

Next, we define $\wt{0^{*_f}_{A, h}}$ to be the $W_hA$-submodule of $H^{d}_\m(W_hA)$ consisting of all elements $\alpha \in H^{d}_\m(W_hA)$ for which there exists $c \in A^{\circ, f}$ such that $F^e_*[cf^{p^e-1}] F^e(\alpha)=0$ for all large $e$. 
\begin{claim}\label{cl:inv2}
If $\alpha \in \wt{0^{*_f}_{A,h}}$ and $[f]\alpha=0$, then $R^{h-1}(\alpha) = 0$.
\end{claim}

\begin{proof}[Proof of Claim~\ref{cl:inv2}]
By definition, there exists an element $c \in A^{\circ, f}$ such that $F^e_*[cf^{p^e-1}]F^e(\alpha)=0$ for all large $e$. 
Since $W_hA$ is Cohen--Macaulay by Proposition~\ref{Witt CM}, we have the following commutative diagram with exact rows:
\[
\begin{tikzcd}[column sep=2.25em]
0 \arrow[r] 
& H^{d-1}_\m(W_hA/([f])) \arrow[r,"\sigma_h"] \arrow[d, "F^e_*{[c]}F^e"] 
& H^d_\m(W_hA) \arrow[r, "\cdot {[f]}"] \arrow[d, "F^e_*{[cf^{p^e-1}]}F^e"] 
& H^d_\m(W_hA) \arrow[d, "F^e_*{[c]}F^e"] \arrow[r] & 0 \\
0 \arrow[r] 
& H^{d-1}_\m(F^e_*(W_hA/([f]))) \arrow[r,"F^e_*\sigma_h"] 
& H^d_\m(F^e_*W_hA) \arrow[r,"\cdot F^e_*{[f]}"]  
& H^d_\m(F^e_*W_hA) \arrow[r] & 0.
\end{tikzcd}
\]
By the assumption that $[f]\alpha=0$, there exists an element $\beta \in H^{d-1}_\m(W_hA/([f]))$ such that $\sigma_h(\beta) = \alpha$. 
The commutativity of the above diagram gives
\[
F^e_*\sigma_h(F^e_*[c]F^e(\beta))
= F^e_*[cf^{p^e-1}]F^e(\sigma_h(\beta))
= F^e_*[cf^{p^e-1}]F^e(\alpha)
= 0
\]
for all large $e$. 
It follows from the injectivity of $F^e_*\sigma_h$ that $\beta \in \wt{0^*_{A,f,h}}$.
Let $\sigma_1 \colon H^{d-1}_\m(A/(f)) \to H^d_\m(A)$ denote the connecting homomorphism induced by $0\to A \xrightarrow{\cdot f} A \to A/(f)\to 0$.
By Claim~\ref{cl:inv1} and the compatibility of connecting homomorphisms with restriction maps, we conclude that 
\[
R^{h-1}(\alpha)
=R^{h-1}(\sigma_h(\beta))
=\sigma_1(R^{h-1}_f(\beta))
=0.  
\]
\end{proof}

Since $\omega_A/\tau_h(\omega_A)$ has finite length,
there exists an integer $r \ge 1$ such that $f^r \omega_A \subseteq \tau_h(\omega_A)$, 
which implies that $f^r 0^*_{A, h} = 0$. 
\begin{claim}\label{cl:inv3}
$R^{h-1}([f^{2r}]\wt{0^*_{A, h+m}}) = 0$ for all integers $m \ge 0$. 
\end{claim}
\begin{proof}[Proof of Claim~\ref{cl:inv3}]
We prove this by induction on $m$. 
The case $m=0$ is immediate, because $R^{h-1}([f^{2r}]\wt{0^*_{A, h}})=f^{2r}0^*_{A, h}=0$. 
Suppose that $m \ge 1$, and let $\alpha \in \wt{0^*_{A, h+m}}$.
Then $R^{h+m-1}([f^r]\alpha) \in f^r 0^*_{A, h}=0$, so 
there exists $\beta \in H^d_\m(W_{h+m-1}A)$ such that 
$V(F_*\beta) = [f^r]\alpha$.
It follows from \cite{KTTWYY3}*{Proposition~3.20} that 
$F_*\beta \in V^{-1}(\wt{0^*_{A, h+m}})=F_*\wt{0^*_{A, h+m-1}}$, and hence $\beta \in \wt{0^*_{A, h+m-1}}$. 
The induction hypothesis yields 
$R^{h-1}([f^{2r}]\beta) = 0$, and in particular $R^{h-1}([f^{pr}]\beta)=0$. 
Consequently, 
\[
R^{h-1}([f^{2r}]\alpha)=R^{h-1}([f^{r}]V(F_*\beta))
=V(F_*(R^{h-1}([f^{pr}]\beta)))
= 0.
\]
\end{proof}

By \cite{KTTWYY3}*{Proposition~3.20} again, we have $R(\wt{0^*_{A, n+1}}) \subseteq \wt{0^*_{A, n}}$. 
Thus $\{\wt{0^*_{A,n}}\}_{n \ge 1}$ forms an inverse system, and we define
\[
\wt{0^*_{A,\infty}}:=\varprojlim_{n \geq 1} \wt{0^*_{A,n}} \subseteq \varprojlim_{n \geq 1} H^d_\m(W_nA). 
\]
It follows from Claim~\ref{cl:inv3} that $[f^{2r}]\wt{0^*_{A,\infty}} = 0$.  
Moreover, since $A$ is Cohen--Macaulay and quasi-$F$-injective, Proposition~\ref{cor:qFinj-to-qF^einj} shows that 
the Frobenius-induced map
\[
F \colon \varprojlim_{n \geq 1} H^d_\m(W_nA) \to \varprojlim_{n \geq 1} F_*H^d_\m(W_nA)
\]
is injective.
Noting that $p^r \ge 2r$, we have $F^{r}([f]\wt{0^*_{A, \infty}}) \subseteq [f^{2r}]\wt{0^*_{A, \infty}} = 0$. 
It follows from the injectivity of $F^r$ that $[f]\wt{0^*_{A,\infty}} = 0$. 

Finally, we show that $\tau_h(\omega_{A})=\omega_A$, equivalently $0^*_{A,h}=0$. 
Let $\alpha \in 0^*_{A,h}$. 
We have $ 0^*_{A,h} = 0^*_{A,n}$ for every integer $n \geq h$, because $\tau_h(\omega_A) = \tau^q(\omega_A)$. 
Thus, for each $n \ge h$, the restriction map induces a surjection 
\[
R^{n-1} \colon \wt{0^*_{A,n}} \twoheadrightarrow 0^*_{A,h}. 
\]
Since the kernel of this map is Artinian, the Mittag--Leffler condition for inverse systems (see \cite{stacks-project}*{Tag~0594}) yields a surjection 
\[
\wt{0^*_{A,\infty}}=\varprojlim_{n \geq 1} \wt{0^*_{A,n}} \twoheadrightarrow 0^*_{A,h}. 
\]
In particular, there exists an element $\alpha_\infty \in \wt{0^*_{A,\infty}}$ mapping to $\alpha$. 
We have shown above that $[f] \wt{0^*_{A,\infty}} = 0$, and hence  $[f]\alpha_\infty=0$. 
Let $\alpha_h$ be the image of $\alpha_{\infty}$ under the truncation map 
\[
\varprojlim_{n \geq 1} H^d_\m(W_nA) \to H^d_\m(W_hA).
\]
Then $[f]\alpha_h = 0$ and $\alpha_h \in \wt{0^*_{A,h}}$. 
By the definition of $\wt{0^*_{A,h}}$, there exists $c \in A^{\circ}$ such that $\alpha_h \in \wt{K^{e,c}_h}$ for all large $e$. 
Since $A/(f)$ is quasi-$F$-rational, $(f)$ is a prime ideal, and we can write $c=f^s c'$ with $s \ge 0$ and $c'\in A^{\circ,f}$. 
Then, for sufficiently large $e$, 
\[
F^e_*[c'f^{p^e-1}] F^e(\alpha_h)=F^e_*[cf^{p^e-s-1}] F^e(\alpha_h)=0.
\] 
Thus $\alpha_h\in \wt{0^{*_f}_{A,h}}$. 
Together with $[f]\alpha_h=0$, Claim~\ref{cl:inv2} gives 
$\alpha = R^{h-1}(\alpha_h) = 0$. 
\end{proof}

\section{Quasi-\texorpdfstring{$F$}{F}-singularities vs.~singularities in birational geometry}
In this section, we discuss the relationship between quasi-$F$-singularities and singularities in birational geometry. 

\subsection{From quasi-$F$-singularities to singularities arising in birational geometry}

The following theorem relates quasi-$F$-singularities to singularities in birational geometry, in parallel with the corresponding results for classical $F$-singularities. 
It generalizes Theorems~\ref{Finjective=>DB} and~\ref{Fpure=>lc}.

\begin{theorem}\label{qFreg-to-klt}
With the notation as in Setting~\ref{notation:qFinj}, assume that $K_A+\Delta$ is $\Q$-Cartier. 
\begin{enumerate}
    \item[(1)] $($\cite{KTTWYY3}*{Theorem~A}$)$ If $(A,\Delta)$ is quasi-$F$-regular, then the pair $(\Spec A,\Delta)$ is klt.
    \item[(2)] $($\cite{STY}*{Theorem~A}$)$ If $(A,\Delta)$ is purely quasi-$F^{\infty}$-split, then the pair $(\Spec A,\Delta)$ is lc. 
    \item[(3)] $($\cite{KTTWYY3}*{Corollary~3.15}$)$ If $A$ is quasi-$F$-rational, then $\Spec A$ has pseudo-rational singularities. 
\end{enumerate}
\end{theorem}

\begin{remark}
Since quasi-$F$-splitting is equivalent to quasi-$F^{\infty}$-splitting in the Gorenstein case by Remark~\ref{remark:quasi-F-singularities}, Theorem~\ref{qFreg-to-klt}(2) shows that Gorenstein quasi-$F$-split singularities are lc.
\end{remark}

Next, we show that Cohen--Macaulay normal quasi-$F$-injective singularities are pseudo-Du~Bois. 

\begin{proposition}\label{qFinj-to-Du-Bois}
With the notation as in Setting~\ref{notation:qFinj}, assume that $A$ is Cohen--Macaulay and essentially of finite type over a field\footnote{The hypothesis that $A$ is essentially of finite type over a field is used only to ensure the existence of a regular alteration. In particular, it can be omitted in dimension at most three.}.
If $A$ is quasi-$F$-injective, then $\Spec A$ has pseudo-Du~Bois singularities. 
\end{proposition}

\begin{proof}
Let $\pi \colon Y \to X:=\Spec A$ be a projective birational morphism from a normal scheme $Y$.
By the proof of \cite{STY}*{Corollary~7.3}, after replacing $\pi$ by a higher birational model, we may assume that $Y$ admits a $\pi$-ample, $\pi$-exceptional Cartier divisor $H$.
Replacing $H$ by a suitable multiple, we may further assume that
$R\pi_*\mathcal{O}_Y(p^iH) \simeq \pi_*\mathcal{O}_Y(p^iH)$ for all $i \geq 0$. 
Using this and the exact sequence
\[
0 \to F^{n}_*\mathcal{O}_Y(p^nH) \xrightarrow{V^n} W_{n+1}\mathcal{O}_Y(H) \xrightarrow{R} W_{n}\mathcal{O}_Y(H) \to 0, 
\]
we obtain an isomorphism
\[
R\pi_*W_n\mathcal{O}_Y(H) \simeq \pi_*W_n\mathcal{O}_Y(H)
\]
for every integer $n \ge 1$. 

Let $E$ denote the reduced exceptional divisor of $\pi$, and fix an integer $e \ge 1$ sufficiently large so that $-p^eE \le H$.
Let $U \subseteq X$ be a big open subset over which $\pi$ is an isomorphism, and let $i \colon U \hookrightarrow X$ be the open immersion. 
Since $\pi$ is an isomorphism over $U$, we identify
$(W_n\omega_Y(E))|_{\pi^{-1}(U)} \simeq W_n\omega_X|_U$. 
Since $W_n\omega_X$ satisfies $(S_2)$ by \cite{STY}*{Proposition~2.26}, this identification gives a natural map
\[
\operatorname{Tr}^E_{\pi,n} \colon \pi_*W_n\omega_Y(E) \to i_*\left((\pi_*W_n\omega_Y(E))|_U\right) \cong i_*(W_n\omega_X|_U) \cong W_n\omega_A. 
\]
We can similarly define $\operatorname{Tr}_{\pi,n} \colon \pi_*W_n\omega_Y \to W_n\omega_A$. 
The horizontal maps in the following diagram are the Frobenius trace maps on Witt sheaves, and the other maps are induced by the natural inclusions and by $\operatorname{Tr}_{\pi,n}$ and $\operatorname{Tr}^E_{\pi,n}$.
\[
\begin{tikzcd}
  W_n\omega_A 
  & F^e_*W_n\omega_A \arrow[l]  \\
& F^e_*\pi_*W_n\omega_Y \arrow[u, "F^e_*\operatorname{Tr}_{\pi,n}"'] \arrow[d] \\
\pi_*W_n\omega_Y(E) \arrow[uu, "\operatorname{Tr}^E_{\pi,n}"]
& F^e_*\pi_*W_n\omega_Y(p^eE) \arrow[l]  
\end{tikzcd}
\]
The commutativity follows after restricting to $U$.
By Lemma~\ref{lem:local duality}, we have 
\begin{align*}
H^d_\m(W_nA)^\vee
&\simeq
(W_n\omega_A)^\wedge, \\
H^d_\m(W_n\mathcal O_Y(-E))^\vee
&\simeq
\Gamma\left(Y,
\mathcal{H}om_{W_n\mathcal O_Y}
(W_n\mathcal O_Y(-E),W_n\omega_Y)
\right)^\wedge \\
&\simeq
\Gamma(Y,W_n\omega_Y(E))^\wedge \\
&\simeq
(\pi_* W_n\omega_Y(E))^\wedge.
\end{align*}
Under these identifications, the Matlis dual of the completion of $\operatorname{Tr}^E_{\pi,n}$ gives a map
\[
\psi_n\colon H^d_\m(W_nA)\to H^d_\m(W_n\mathcal O_Y(-E)).
\]
Taking the Matlis dual of the above diagram in the sense of Lemma~\ref{lem:local duality}, we obtain the following commutative diagram: 
\begin{equation}\label{eq:diag-ap}
\begin{tikzcd}[column sep=1.5em]
H^d_\m(W_nA) \arrow[r,"F^e"] \arrow[dd, "\psi_n"']
  & H^d_\m(F^e_*W_nA) \arrow[d] \\
& H^d_\m(F^e_*W_n\mathcal{O}_Y)  \\
H^d_\m(W_n\mathcal{O}_Y(-E)) \arrow[r,"F^e"] 
  & H^d_\m(F^e_*W_n\mathcal{O}_Y(-p^eE)) \arrow[u] \arrow[uu, bend right=70,"\varphi"']
\end{tikzcd}
\end{equation}
where the map $\varphi$ exists by the claim below.

\begin{claim*}
In the bounded derived category of $W_nA$-modules, the morphism  
\[
R\pi_*W_n\mathcal{O}_Y(-p^eE) \longrightarrow R\pi_*W_n\mathcal{O}_Y
\]
induced by the inclusion
$W_n\mathcal{O}_Y(-p^eE) \subseteq W_n\mathcal{O}_Y$
factors through
$\pi_*W_n\mathcal{O}_Y \cong W_nA$.
\end{claim*}

\begin{proof}[Proof of Claim]
Since $R\pi_*W_n\mathcal{O}_Y(H)\simeq \pi_*W_n\mathcal{O}_Y(H)$, 
the inclusions 
\[
W_n\mathcal{O}_Y(-p^eE) \subseteq W_n\mathcal{O}_Y(H) \subseteq W_n\mathcal{O}_Y
\]
induce a morphism $R\pi_*W_n\mathcal{O}_Y(-p^eE) \to R\pi_*W_n\mathcal{O}_Y$ 
that factors through
\[\pi_*W_n\mathcal{O}_Y(H)\to \pi_*W_n\mathcal{O}_Y \cong W_nA.\]
\end{proof}
By the claim, after applying the exact functor $F^e_*$, the morphism 
\[
R\pi_*F^e_*W_n\mathcal O_Y(-p^eE)
\longrightarrow
R\pi_*F^e_*W_n\mathcal O_Y
\]
induced by the inclusion factors through $F^e_*W_nA$. 
Applying $H^d(R\Gamma_{\m_n}(-))$ to this gives a map
\[
\varphi\colon
H^d_\m(F^e_*W_n\mathcal O_Y(-p^eE))
\longrightarrow
H^d_\m(F^e_*W_nA),
\]
and the resulting triangle is commutative.

It follows from diagram \eqref{eq:diag-ap} that the composition
\[
H^d_\m(W_nA) \xrightarrow{\psi_n} 
H^d_\m(W_n\mathcal{O}_Y(-E)) 
\longrightarrow 
H^d_\m(F^e_*W_nA).
\]
is $F^e$. 
In particular, we have $\Ker(\psi_n) \subseteq K^e_n$. 
By the proof of Proposition~\ref{cor:qFinj-to-qF^einj}, there exists an integer $n_0 \ge 1$ such that
$R^{n_0-1}(K^e_{n_0}) = 0$, so $R^{n_0-1}(\Ker(\psi_{n_0})) = 0$.
Now we consider the following commutative diagram with exact rows:
\[
\begin{tikzcd}[column sep=small]
H^d_\m(F_*W_{n_0-1}A) \arrow[r, "V"] \arrow[d,twoheadrightarrow,"\psi'"] 
  & H^d_\m(W_{n_0}A) \arrow[r, "R^{n_0-1}"] \arrow[d,twoheadrightarrow,"\psi_{n_0}"] 
  & H^d_\m(A) \arrow[r] \arrow[d,twoheadrightarrow,"\psi_1"] & 0 \\
H^d_\m(F_*W_{n_0-1}\mathcal{O}_Y(-pE)) \arrow[r, "V"] 
  & H^d_\m(W_{n_0}\mathcal{O}_Y(-E)) \arrow[r] 
  & H^d_\m(\mathcal{O}_Y(-E)) \arrow[r] & 0,
\end{tikzcd}
\]
where $\psi'$ is the Matlis dual, in the sense of Lemma~\ref{lem:local duality}, of the completion of the map 
\[
F_*\pi_*W_{n_0-1}\omega_Y(pE) \longrightarrow F_*W_{n_0-1}\omega_A.
\]
A diagram chase, together with the vanishing $R^{n_0-1}(\Ker(\psi_{n_0}))=0$, shows that 
\[
\Ker(\psi_1) = R^{n_0-1}(\Ker(\psi_{n_0})) = 0.
\]
Taking the Matlis dual of $\psi_1$, we see that the natural map $\pi_*\omega_Y(E) \to \omega_A$ is surjective. 
Hence $\pi_*\omega_Y(E)=\omega_A$, and thus $A$ is pseudo-Du~Bois. 
\end{proof}

\subsection{From singularities arising in birational geometry to quasi-$F$-singularities}

In dimension two, quasi-$F$-singularities provide a closer positive-characteristic counterpart to singularities in birational geometry than classical $F$-singularities do. 
The following theorem shows, for example, that klt singularities are characterized by quasi-$F$-regularity, and that lc singularities are characterized by quasi-$F$-splitting under mild hypotheses.

\begin{theorem}\label{klt-to-qFreg}
With the notation as in Setting~\ref{notation:qFinj}, assume that $d = 2$.
\begin{enumerate}
    \item[(1)] $($\cite{KTTWYY3}*{Theorem~C}$)$ The pair $(A,\Delta)$ is quasi-$F$-regular if and only if $(\Spec A, \Delta)$ is klt. 
    \item[(2)] $($\cite{STY}*{Theorem~B}$)$ Suppose that the residue field $A/\m$ is perfect and $K_A+\Delta$ is $\Q$-Cartier with index not divisible by $p$. 
    Then the pair $(A,\Delta)$ is purely quasi-$F^{\infty}$-split if and only if $(\Spec A, \Delta)$ is lc. 
    \item[(3)] $($\cite{STY}*{Theorem~C}$)$
Suppose that the residue field $A/\m$ is perfect. 
Then $A$ is quasi-$F$-split\footnote{In dimension two, quasi-$F$-splitting is equivalent to quasi-$F^{\infty}$-splitting by Remark~\ref{remark:quasi-F-singularities}(1).} if and only if one of the following conditions holds: 
\begin{enumerate}
\item[(a)] $\Spec A$ is klt, 
\item[(b)] $\Spec A$ is lc and the Gorenstein index of $A$ is not divisible by $p$. 
\end{enumerate}
In particular, when $p>3$, quasi-$F$-splitting, quasi-$F^{\infty}$-splitting and lc singularities are equivalent. 
\end{enumerate}
\end{theorem}

We now turn to pseudo-rational and pseudo-Du~Bois singularities.
We first prove the following result on pseudo-rational singularities.
\begin{theorem}\label{rational-to-qFrat}
With the notation as in Setting~\ref{notation:qFinj}, assume that $d = 2$.
Then $A$ is quasi-$F$-rational if and only if $\Spec A$ has pseudo-rational singularities. 
\end{theorem}

\begin{proof}
We will show the ``if" direction. 
Let $\pi \colon Y \to X := \Spec A$ be a log resolution and $D=\operatorname{div}_X(c)$ be an arbitrary effective Cartier divisor on $X$. 
Since $Q_{A,n}^e(c) \cong Q^e_{A,(1/p^e)D,n}$, 
it suffices to show that the natural map
\[
H^d_\m(A) \longrightarrow H^d_\m(Q^e_{A,(1/p^e)D,n})
\]
is injective for some integers $e, n \ge 1$.

Let $E$ be an effective $\pi$-anti-ample exceptional Cartier divisor on $Y$.
Choose an integer $e$ large enough so that 
\[
\left\lfloor \frac{1}{p^e}(E + \pi^*D) \right\rfloor = 0, 
\]
and set $\Delta := ({1}/{p^e})D$ and $B_Y := ({1}/{p^e})(E + \pi^*D)$.
Note that $B_Y$ is $\pi$-anti-ample.
Since $A$ has rational singularities, the natural map $\pi_*\omega_Y \to \omega_X$ is an isomorphism. 
Taking its Matlis dual, we obtain an isomorphism 
\[
H^2_\m(A) \xrightarrow{\ \simeq\ }  H^2_\m(\mathcal{O}_Y)=H^2_\m(\mathcal{O}_Y(B_Y)).
\]

Now consider the following commutative diagram:
\[
\begin{tikzcd}
A \arrow[r] \arrow[d] 
  & Q^e_{A,\Delta,n} \arrow[d] \\
R\pi_*\mathcal{O}_Y(B_Y) \arrow[r] 
  & R\pi_*Q^e_{Y,B_Y,n}.
\end{tikzcd}
\]
Accordingly, the above diagram induces the following commutative diagram on local cohomology:
\[
\begin{tikzcd}
H^2_\m(A) \arrow[r] \arrow[d,"\simeq"'] 
  & H^2_\m(Q^e_{A,\Delta,n}) \arrow[d] \\
H^2_\m(\mathcal{O}_Y(B_Y)) \arrow[r] 
  & H^2_\m(Q^e_{Y,B_Y,n}).
\end{tikzcd}
\]
Therefore, it suffices to prove that the bottom horizontal map is injective. 
To this end, we show that $H^1_\m(B^e_{Y,B_Y,n}) = 0$, where $B^e_{Y,B_Y,n}$ denotes the cokernel of the map $\sO_Y(B_Y) \to Q^e_{Y,B_Y,n}$. 
By an argument similar to the proof of \cite{TWY}*{Theorem~3.27} (see also \cite{KTTWYY1}*{Theorem~5.13} for the case $e=1$), this vanishing can be verified once the following conditions are satisfied: 
\begin{enumerate}
    \item[(i)] $H^0_{\m}(Y, \Omega^1_Y(\log E)(p^cB_Y)) = 0$ for every $0 \le c \le e-1$;
    \item[(ii)] $H^0_{\m}(Y, B_1\Omega_Y^2(\log E)(p^kB_Y)) = 0$ for every $k \ge 1$; and
    \item[(iii)] $H^1(Y, \Omega_Y^1(\log E)^* \otimes \mathcal{O}_Y(K_Y - p^{n+c}B_Y)) = 0$ for every $0 \le c \le e-1$, where
    $\Omega^1_Y(\log E)^* := \mathcal{H}\!om_{\mathcal{O}_Y}(\Omega_Y^1(\log E), \mathcal{O}_Y)$.
\end{enumerate}
Conditions (i) and (ii) follow from the torsion-freeness of the sheaves involved, and (iii) holds by Serre vanishing for sufficiently large $n$.
This proves the desired vanishing and completes the proof.
\end{proof}

\begin{lemma}\label{qFinj-resolution}
With the notation as in Setting~\ref{notation:qFinj}, assume that $d = 2$.
Let $\pi \colon Y \to \Spec A$ be a log resolution and $E$ be its reduced exceptional divisor.  
\begin{enumerate}
  \item[(1)] The Frobenius-induced map
  \[
  F \colon \varprojlim_n H^2_\m(W_n\mathcal{O}_Y) \longrightarrow \varprojlim_n F_*H^2_\m(W_n\mathcal{O}_Y)
  \]
  is injective.

  \item[(2)] If $k=A/\m$ is perfect, then the Frobenius-induced map
  \[
  F \colon \varprojlim_n H^1_\m(W_n\mathcal{O}_E) \longrightarrow \varprojlim_n F_*H^1_\m(W_n\mathcal{O}_E)
  \]
  is injective.
\end{enumerate}
\end{lemma}

\begin{proof}
(1)  
Set $B := (1/2)E$, which is $\pi$-anti-ample.  
By the case $e=1$ of the argument used in the proof of Theorem~\ref{rational-to-qFrat} (see also \cite{KTTWYY1}*{Theorem~5.13}),
\[
H^2_\m(\mathcal{O}_Y(B)) \longrightarrow H^2_\m(Q_{Y,B,m})
\]
is injective for some integer $m \ge 1$.  
Therefore, the commutative diagram
\[
\begin{tikzcd}
    H^2_\m(\mathcal{O}_Y) \arrow[r] \arrow[d,equal] &
    H^2_\m(Q_{Y,m}) \arrow[d] \\
    H^2_\m(\mathcal{O}_Y(B)) \arrow[r] &
    H^2_\m(Q_{Y,B,m}),
\end{tikzcd}
\]
shows that the top horizontal map $H^2_\m(\mathcal{O}_Y) \longrightarrow H^2_\m(Q_{Y,m})$ is injective. 
Set
\[
K_n^Y:=\Ker\left(H^2_{\m}(W_n\mathcal O_Y)\xrightarrow{F}
F_*H^2_{\m}(W_n\mathcal O_Y)\right)
\]
for each integer $n\ge 1$. 
The injectivity of $H^2_\m(\mathcal{O}_Y) \longrightarrow H^2_\m(Q_{Y,m})$
means that the restriction map $K_m^Y\to K_1^Y$ is zero. 
Moreover, Grauert–Riemenschneider vanishing gives $R^1\pi_*\omega_Y=0$, and hence $H^1_\m(\mathcal O_Y)=0$. 
Therefore the exact sequences
\[
0\to F_*H^2_\m(W_{r-1}\mathcal O_Y)
\xrightarrow{V}
H^2_\m(W_r\mathcal O_Y)
\xrightarrow{R^{r-1}}
H^2_\m(\mathcal O_Y)\to 0
\]
are available for all $r\ge 1$. 
Using these exact sequences and the relations among $F$, $V$ and $R$, the same induction argument as in the proof of Proposition~\ref{cor:qFinj-to-qF^einj} shows that the map $K_{m+n}^Y\to K_{n+1}^Y$ is zero for every integer $n\ge 0$. 
Passing to the inverse limit, we obtain $\varprojlim_n K_n^Y=0$. 
This is precisely the desired injectivity of
\[
F \colon \varprojlim_n H^2_\m(W_n\mathcal O_Y)
\to
\varprojlim_n F_*H^2_\m(W_n\mathcal O_Y).
\]

(2)  
For an $\F_p$-scheme $X$, we define the $W_n\mathcal{O}_X$-module $B_{X,n}$ to be the cokernel of the Frobenius map $F \colon W_n \sO_X \to F_*W_n\sO_X$.  
It then suffices to show that $\varprojlim_n H^0(E, B_{E,n}) = 0$.  
Write $E=\sum_{i=1}^r E_i $ for the decomposition into irreducible components.  
By the proof of \cite{STY}*{Claim~6.11}\footnote{In \cite{STY}*{Claim~6.11}, the residue field is assumed to be infinite, but this assumption is unnecessary in the no boundary case.}, we have an isomorphism 
\[
B_{E,n} \simeq \bigoplus_{i=1}^r B_{E_i,n}.
\]
Since each $E_i$ is a smooth projective curve over a perfect field $k$, we can consider the Cartier operator $C$ on $E_i$: 
\[
0 \to B_1 \Omega_{E_i} \to F_*\Omega_{E_i} \xrightarrow{C} \Omega_{E_i} \to 0,
\]
where $B_1 \Omega_{E_i}$ is the image of $d \colon F_*\sO_{E_i} \to F_*\Omega_{E_i}$. 
Setting 
\[
C^n:=C \circ F_*C \circ \cdots \circ F^{n-1}_*C \colon F^n_*\Omega_{E_i} \to \Omega_{E_i}, 
\] 
we define the $\sO_{E_i}$-submodule $B_n\Omega_{E_i}$ of $F^n_*\Omega_{E_i}$ by $B_n\Omega_{E_i}=\ker C^n$. 
The $k$-vector space $H^0(E_i, \Omega_{E_i})$ is finite-dimensional, so the increasing sequence 
\[\{H^0(E_i, B_n\Omega_{E_i})\}_{n \ge 1}\] 
stabilizes. 
That is, there exists an integer $N \ge 1$ such that 
\[
V:=H^0(E_i, B_N\Omega_{E_i})=H^0(E_i, B_n\Omega_{E_i})
\]
for all $n\geq N$.
Moreover, by a result of Serre (cf.~ 
\cite{Serre58}*{Lemme 2 in Section 7} or 
\cite{illusie_de_rham_witt}*{Ch.\ I, Remarques 3.12(a)}), there exists an isomorphism
\[
B_{E_i,n} \cong B_n\Omega_{E_i}. 
\]
Under this isomorphism, the transition map $B_{E_i,n+1}\to B_{E_i,n}$ corresponds to the map induced by the Cartier operator
\[
F^n_*C \colon B_{n+1}\Omega_{E_i}\to B_n\Omega_{E_i}.
\]
Since $V=H^0(E_i,B_N\Omega_{E_i})$ is killed by the map induced by $C^N$, the $N$-th iterate of the transition map is zero on $V$. 
Hence $\varprojlim_n H^0(E_i,B_{E_i,n})=0$ for all $i=1, \dots, r$, which implies that $\varprojlim_n H^0(E,B_{E,n})=0$. 
\end{proof}

\begin{theorem}\label{thm:DuBois-to-qFinj}
With the notation as in Setting~\ref{notation:qFinj}, assume that $d = 2$ and the residue field $A/\m$ is perfect.
Then $A$ is quasi-$F$-injective if and only if $\Spec A$ has pseudo-Du~Bois singularities. 
\end{theorem}

\begin{proof}
Proposition~\ref{qFinj-to-Du-Bois} proves the ``only if'' direction.
It remains to prove the ``if'' direction. 
Let $\pi \colon Y \to \Spec A$ be a log resolution and $E$ be the reduced exceptional divisor of $\pi$. 
Let $\mathcal{I}_E \subseteq \sO_Y$ denote the ideal sheaf of $E$. 
The $W_n\mathcal{O}_Y$-submodule $W_n\mathcal{I}_E$ of the constant sheaf $W_n\mathscr{K}_Y$ is defined as follows: 
for each open subset $U \subseteq Y$, we set 
\begin{align*}
\hspace*{3em} \Gamma(U, W_n\mathcal{I}_E)
&\coloneqq 
\bigl\{\,(\varphi_0,\ldots,\varphi_{n-1}) 
 \bigm| 
\varphi_i \in \Gamma(U,\mathcal{I}_E)
\text{ for all } i \,\bigr\} \\
&\subseteq \Gamma(U, W_n\mathscr{K}_Y).
\end{align*}
We define the $W_n\sO_Y$-module $Q^{E}_{Y,n}$ and the $W_n\sO_Y$-linear map $\Phi^E_{Y,n} \colon \mathcal{I}_E \to Q^{E}_{Y,n}$ by the pushout diagram
\[
\begin{tikzcd}
W_n\mathcal{I}_E \arrow[r,"F"] \arrow[d,"R^{n-1}"'] &
F_*\bigl(W_n\mathcal{I}_E \bigr) \arrow[d] \\
\mathcal{I}_E \arrow[r,"\Phi^{E}_{Y,n}"] &
Q^{E}_{Y,n}, 
\end{tikzcd}
\]
and for each integer $n \ge 1$, set $K^{Y,E}_n:=\Ker\left(H^2_{\m}(W_n\mathcal{I}_E) \xrightarrow{F} F_*H^2_{\m}(W_n\mathcal{I}_E)\right)$. 
We have the following commutative diagram with exact rows:
\[
\begin{tikzcd}[column sep=0.4cm]
    0 \arrow[r] 
      & \varprojlim_n H^1_\m(W_n\mathcal{O}_E) \arrow[r] \arrow[d,"F_E"] 
      & \varprojlim_n H^2_\m(W_n\mathcal{I}_E) \arrow[r] \arrow[d,"F_{Y,E}"] 
      & \varprojlim_n H^2_\m(W_n \mathcal{O}_Y) \arrow[r] \arrow[d,"F_Y"] 
      & 0 \\
    0 \arrow[r] 
      & \varprojlim_n F_*H^1_\m(W_n\mathcal{O}_E) \arrow[r] 
      & \varprojlim_n F_*H^2_\m(W_n\mathcal{I}_E) \arrow[r] 
      & \varprojlim_n F_*H^2_\m(W_n\mathcal{O}_Y) \arrow[r] 
      & 0. 
\end{tikzcd}
\]
Since $F_Y$ and $F_E$ are injective by Lemma~\ref{qFinj-resolution}, the Snake Lemma shows that $F_{Y,E}$ is also injective, 
and the injectivity of $F_{Y,E}$ implies that $\varprojlim_n K^{Y,E}_n=0$. 
Then, by an argument similar to the proof of Claim~\ref{prop:infty-qFinj}, the truncation $K^{Y,E}_n \to K^{Y,E}_1$ is the zero map for some $n \ge 1$, 
which is equivalent to the injectivity of the map 
\[
H^2_\m(\mathcal{I}_E) \longrightarrow H^2_\m(Q^E_{Y,n}). 
\]

Finally, by the universal property of the pushout obtained by applying $H^2_\m(-)$ to the defining pushout diagram of $Q_{A,n}$, we have the following commutative diagram:
\[
\begin{tikzcd}
H^2_\m(A) \arrow[r] \arrow[d,"\simeq"'] 
  & H^2_\m(Q_{A,n}) \arrow[d] \\
H^2_\m(\mathcal{I}_E) \arrow[r,hookrightarrow] 
  & H^2_\m(Q^E_{Y,n}),
\end{tikzcd}
\]
where the left vertical isomorphism follows from the assumption that $\Spec A$ has pseudo-Du~Bois singularities.
Therefore, the top horizontal map 
\[
H^2_\m(A) \longrightarrow H^2_\m(Q_{A,n}) 
\]
is injective. 
\end{proof}

We conclude this subsection with a remark on the higher-dimensional case. 
Cascini--Tanaka--Witaszek \cite{CTW15a} showed that for every prime number $p$, there exists a three-dimensional canonical singularity over an algebraically closed field of characteristic $p$ that is not $F$-pure.
On the other hand, we have the following theorem.

\begin{theorem}[\textup{\cite{KTTWYY3}*{Theorem B}; see also \cite{KTTWYY1}, \cite{KTTWYY2}}]\label{thm:quasi-F-split of 3-dim klt}
Let $A$ be a 3-dimensional $\Q$-factorial normal local ring essentially of finite type over a perfect field of characteristic $p>41$.
If $\Spec A$ has klt singularities, then $A$ is quasi-$F$-regular. 
\end{theorem}

\begin{remark}\label{remark:higher-dimensional case}
We collect some related results and consequences. 
\begin{enumerate}
    \item The assumption on $p$ is optimal: 
    there exists a 3-dimensional $\Q$-factorial klt singularity over a perfect field of characteristic $41$ that is not quasi-$F$-split (see \cite{KTTWYY2}*{Theorem B}). 
    \item For each prime $p$, there exists a 4-dimensional $\Q$-factorial canonical singularity that is not quasi-$F$-split (see \cite{TY26}).
    \item By considering the affine cone over a supersingular K3 or abelian surface, 
    one can construct a 3-dimensional non-quasi-$F$-split lc singularity in every characteristic $p>0$. 
    See Example~\ref{example:non-qFs} for an explicit example of a non-quasi-$F$-split lc singularity. 
    \item  
    Combining Theorem~\ref{thm:quasi-F-split of 3-dim klt} and \cite{Kaw4}, we can prove the logarithmic extension theorem for one-forms for klt threefolds over a perfect field of characteristic $p>41$ (see \cite{KTTWYY2}, \cite{Kawakami-Sato}). 
    We can also obtain the Steenbrink vanishing theorem for them as an application of Theorem~\ref{thm:quasi-F-split of 3-dim klt} (see \cite{Kaw7}). 
\end{enumerate}
\end{remark}

\section{Fedder-type criteria for quasi-\texorpdfstring{$F$}{F}-singularities}

As discussed above, 
quasi-$F$-singularities provide a closer positive-characteristic counterpart to singularities in birational geometry than classical $F$-singularities do. 
Thus, it is natural to ask how a given singularity can be classified in terms of quasi-$F$-singularities. 
When the singularity is a hypersurface, we discuss criteria analogous to Propositions~\ref{Fedder-Fpure} and~\ref{Fedder-fpt}.

\subsection{A criterion for quasi-$F$-splitting}
In this subsection, we review a criterion for quasi-$F$-splitting that generalizes Fedder's criterion for $F$-purity (Proposition~\ref{Fedder-Fpure}).
Suppose that $A$ is a hypersurface singularity in characteristic zero. 
If its reduction $A_p$ modulo $p$ is quasi-$F$-split for a single prime $p$, then it follows from \cite{p-pure}*{Proposition~5.13} and \cite{Yoshikawa25}*{Theorem~A} (see also \cite{Zhu17}*{Corollary 4.2} and \cite{ST25}*{Theorem~A}) that $\Spec A$ has lc singularities. 
Thus, checking quasi-$F$-splitting for a small prime $p$ provides a sufficient condition for log canonicity of hypersurface singularities in characteristic zero. 
This illustrates the usefulness of such a criterion. 

\begin{setting}\label{notation:fedder-qFs}
Let $S \coloneqq k[x_1,\ldots,x_N]$ be the polynomial ring in $N$ variables over a perfect field $k$ of characteristic $p>0$. 
Let $\mathfrak{m}\coloneqq (x_1,\ldots,x_N)$ be the maximal ideal of $S$ corresponding to the origin, and let $\m^{[p^n]}$ denote its $n$-th Frobenius power, that is, the ideal $(x_1^{p^n},\ldots,x_N^{p^n})$ for every integer $n \geq 1$. 
Note that $F_*S$ is a free $S$-module with basis 
\[
\bigl\{\,F_*(x_1^{i_1}\cdots x_N^{i_N}) \,\bigm|\, 0 \le i_1,\ldots,i_N \le p-1 \,\bigr\}. 
\]
The \emph{Frobenius trace map} $\mathrm{Tr} \colon F_*S \to S$ is the $S$-module homomorphism that sends   
$F_*((x_1\cdots x_N)^{p-1})$ to $1$ and all other basis elements to zero. 
Moreover, $\Hom_S(F_*S,S)$ is a free $F_*S$-module of rank one generated by $\mathrm{Tr}$. 

Let $f \in \m$ be a nonzero element.
Write $f=\sum_{j=1}^r c_j M_j$, where $c_j \in k^{\times}$ and the \(M_j\) are distinct monomials. 
Set 
\[
\Delta(f)
:= \sum_{\substack{0 \le \alpha_1,\ldots,\alpha_r \le p-1\\ \alpha_1+\cdots+\alpha_r=p}}
\frac{1}{p}\binom{p}{\alpha_1,\ldots,\alpha_r} (c_1M_1)^{\alpha_1}\cdots (c_rM_r)^{\alpha_r},
\]
where \(\frac{1}{p}\binom{p}{\alpha_1,\ldots,\alpha_r}\) denotes the image in \(k\) of the integer
\(\binom{p}{\alpha_1,\ldots,\alpha_r}/p\). 
We then define an $S$-linear map $\theta_f \colon F_*S \longrightarrow S$ as the composite 
\[
\theta_f \colon F_*S \xrightarrow{\cdot F_*(f^{p(p-2)}\Delta(f))} F_*S \xrightarrow{\mathrm{Tr}} S,
\]
where $F_*S \xrightarrow{\cdot F_*(f^{p(p-2)}\Delta(f))} F_*S$ denotes multiplication by $F_*(f^{p(p-2)}\Delta(f))$. 
We also define a sequence of ideals $\{\fI{n}{f}\}_{n\ge1}$ inductively by 
\begin{align*}
\fI{1}{f} 
&\coloneqq (f^{p-1}),\\
\fI{n+1}{f} 
&\coloneqq \theta_f\Bigl(F_*\fI{n}{f} \cap \Ker \mathrm{Tr} \Bigr)+(f^{p-1}).
\end{align*}
We note that the sequence $\{I_n(f)\}_{n \geq 1}$ is increasing.
Indeed, $I_2(f)$ contains $I_1(f)$ by definition.
Suppose that $I_{n-1}(f) \subseteq I_n(f)$. 
Then
\begin{align*}
I_{n+1}(f) 
&=\theta_f\Bigl(F_*\fI{n}{f} \cap \Ker \mathrm{Tr} \Bigr)+(f^{p-1}) \\
&\supseteq  \theta_f\Bigl(F_*\fI{n-1}{f} \cap \Ker \mathrm{Tr} \Bigr)+(f^{p-1})=I_n(f).    
\end{align*}
Therefore, by induction, the sequence $\{I_n(f)\}_{n \geq 1}$ is increasing. 
\end{setting}

\begin{remark}
In \cite{KTY22}*{Theorem~4.11}, the homomorphism $\theta$ is defined by
\[
\theta \colon F_*S \xrightarrow{\cdot F_*\Delta(f^{p-1})} F_*S \xrightarrow{\mathrm{Tr}} S.
\]
On the other hand, \cite{KTY22}*{Proposition~3.8(4)} shows that 
\[
\Delta(f^{p-1})
=
-f^{p(p-2)}\Delta(f)
\quad
\textup{in $S/F(S)$}.
\]
Consequently, $\theta_f$ gives the same sequence of ideals as the one appearing in \cite{KTY22}*{Theorem~4.11}.
For computational purposes, it is more efficient to use
$f^{p(p-2)}\Delta(f)$ instead of $\Delta(f^{p-1})$, since the former is easier to compute.
\end{remark}

\begin{definition}[\cite{yobuko19}]
The \emph{quasi-$F$-split height} $\sht(A)$ of an $F$-finite local domain $A$ is defined as the infimum of integers $n \ge 1$ such that $A$ is $n$-quasi-$F$-split. 
If $A$ is not quasi-$F$-split, we set $\sht(A)=\infty$. 
\end{definition}

The quasi-$F$-split height of the hypersurface  singularity $(S/f)_\m$ can be characterized in terms of the ideals $\{I_n(f)\}_{n \ge 1}$. 
\begin{theorem}[\cite{KTY22}*{Theorem~4.11}]\label{fedder-qFs}
With the notation as in Setting~\ref{notation:fedder-qFs}, we have 
\[
\sht\bigl((S/f)_\m\bigr)
=\inf\bigl\{n \ge 1 \bigm| \fI{n}{f} \nsubseteq \m^{[p]} \bigr\}.
\]
\end{theorem}

\begin{remark}\label{comp-delta}
For each integer $e \ge 1$, set 
\begin{align*}
\mathrm{Tr}^e & \colon F^e_*S \xrightarrow{F^{e-1}_*\mathrm{Tr}} F^{e-1}_*S \to \cdots \to F_*S \xrightarrow{\mathrm{Tr}} S, \\
\theta_f^e & \colon F^e_*S \xrightarrow{F^{e-1}_*\theta_f} F^{e-1}_*S \to \cdots \to F_*S \xrightarrow{\theta_f} S.
\end{align*}
Then, by the definition of the ideals $\{I_n(f)\}_{n \ge 1}$, we have 
\[
\fI{n}{f} \subseteq 
\sum_{i=0}^{n-1} \theta_f^i(F^i_*(f^{p-1}S))
   = \sum_{i=0}^{n-1} \mathrm{Tr}^i\bigl(F^i_*(f^{p-1}(f^{p(p-2)}\Delta(f))^{p^{i-1}+\cdots+1}S)\bigr), 
\]
where the exponent $p^{i-1}+\cdots+1$ is understood to be $0$ when $i=0$. 
Therefore, if $f^{p-1}(f^{p(p-2)}\Delta(f))^{p^{i-1}+\cdots+1} \in \m^{[p^{i+1}]}$ for all $0 \le i \le n-1$, 
then $\fI{n}{f} \subseteq \m^{[p]}$.
Since each term of $\Delta(f)$ is a product of $p$ monomials appearing in $f$, each term of $f^{p-1}(f^{p(p-2)}\Delta(f))^{p^{i-1}+\cdots+1}$ is a product of $p^{i+1}-1$ monomials appearing in $f$. 
It follows that if every product of $p^{i+1}-1$ monomials appearing in $f$ lies in $\m^{[p^{i+1}]}$ for all $0 \le i \le n-1$, then $\fI{n}{f} \subseteq \m^{[p]}$.
\end{remark}

\begin{example}\label{example:fedder-qFs}
We retain the notation of Setting~\ref{notation:fedder-qFs}.
We first recall that, when $p=2$, if $f=T_1+\cdots+T_r$ is written as a sum of distinct terms, then $\Delta(f)=\sum_{1\le i<j\le r} T_iT_j$. 

\textup{(1)}
Let $p=2$, $N=3$ and $f = x^2 + y^3 + z^5$. 
We will show that $\sht((S/f)_\m)=4$, which is equivalent, by Theorem~\ref{fedder-qFs}, to showing that $\fI{3}{f} \subseteq \m^{[2]}$ and $\fI{4}{f} \not\subseteq \m^{[2]}$. 

We first verify that $\fI{3}{f} \subseteq \m^{[2]}$.
Clearly $f \in \m^{[2]}$.
Moreover, every product of $3$ (resp.~$7$) monomials appearing in $f$ lies in $\m^{[4]}$ (resp.~$\m^{[8]}$). 
Thus $\fI{3}{f} \subseteq \m^{[2]}$ by Remark~\ref{comp-delta}. 
It remains to prove that $\fI{4}{f} \not\subseteq \m^{[2]}$.
Since $p=2$, we have
\[
f^{p(p-2)}\Delta(f) = x^2y^3 + x^2z^5 + y^3z^5. 
\]
A direct computation gives
\[
\theta_f(F_*(xf))=xyz^2,
\quad
\theta_f(F_*(xyz^2))=xz^3,
\quad
\theta_f(F_*(xz^3))=xyz.
\]
Moreover,
\[
\mathrm{Tr}(F_*(xf))=\mathrm{Tr}(F_*(xyz^2))=\mathrm{Tr}(F_*(xz^3))=0.
\]
Hence \(xyz^2\in \fI{2}{f}\), \(xz^3\in \fI{3}{f}\) and $xyz \in \fI{4}{f}$. 
Since $xyz \notin \m^{[2]}$, this proves $\fI{4}{f}\not\subseteq \m^{[2]}$. 

\medskip

\textup{(2)}
Let $p=2$, $N=5$ and $f=x_3x_4x_5^2+x_2^3x_4+x_1^3x_3$.
Then $\sht((S/f)_\m)=2$. 
Indeed, $\fI{1}{f}=fS \subseteq \m^{[2]}$.
Set
\[
T_1:=x_3x_4x_5^2,\quad
T_2:=x_2^3x_4,\quad
T_3:=x_1^3x_3.
\]
By the formula for $\Delta(f)$ when $p=2$, we have
\[
f\Delta(f)=(T_1+T_2+T_3)(T_1T_2+T_1T_3+T_2T_3).
\]
All terms in this expansion except $T_1T_2T_3$ belong to $\m^{[4]}$. 
Hence 
\[
f\Delta(f)\equiv T_1T_2T_3=x_1^3x_2^3x_3^2x_4^2x_5^2
\pmod{\m^{[4]}}.
\]
Therefore,
\[
x_3x_4x_5 f\Delta(f)
\equiv
(x_1\cdots x_5)^3
\pmod{\m^{[4]}}.
\]
It follows that
\[
\theta_f(F_*(x_3x_4x_5f))\equiv x_1\cdots x_5 \pmod{\m^{[2]}}.
\]
Since $\mathrm{Tr}(F_*(x_3x_4x_5f))=0$, this shows that $\fI{2}{f} \not\subseteq \m^{[2]}$.

\medskip

\textup{(3)}
Let $p=2$, $N=6$ and $f=x_1x_2x_6^2+x_3x_4x_5^2+x_2^3x_4+x_1^3x_3$.
We will show that $\sht((S/(x_6,f))_\m)=2$ and  $\sht((S/f)_\m)=\infty$. 
In particular, quasi-$F$-splitting does not deform, even for Gorenstein rings. 
Since
\[
S/(x_6,f) \simeq k[x_1,\ldots,x_5]/(x_3x_4x_5^2+x_2^3x_4+x_1^3x_3),
\]
the equality $\sht((S/(x_6,f))_\m)=2$ follows from \textup{(2)}. 

It remains to show that $\sht((S/f)_\m)=\infty$. 
Unlike in the previous examples, the same strategy does not apply in this setting. 
Indeed, 
\[
f \Delta(f) \equiv x_1^3 x_2^3 x_3^2 x_4^2 x_5^2 \pmod{\m^{[4]}},
\]
which implies that 
\[
\mathrm{Tr}(F_*(f \Delta(f) x_3 x_4 x_5 x_6)) = x_1 x_2 x_3 x_4 x_5 \notin \m^{[2]}.
\]
However, $\mathrm{Tr}(F_*(x_3 x_4 x_5 x_6 f)) = x_6 \ne 0$, this does not give an element of $\fI{2}{f}$. 

In order to prove that $\sht((S/f)_\m)=\infty$, it suffices to show that $\fI{n}{f} \subseteq \m^{[2]}$ for every integer $n \geq 1$ by Theorem~\ref{fedder-qFs}. 
A Macaulay2 computation using the script available at \cite{Takamatsu_code} shows that $\fI{4}{f}=\fI{5}{f}$ and that this ideal is generated by the following elements:\footnote{The equality $\sht((S/f)_\m)=\infty$ is also proved in \cite{KTY25}*{Example~4.3} without using Macaulay2.}
\[
\begin{aligned}
&x_2^3x_4x_6 + x_3x_4x_5^2x_6, \quad 
x_1^3x_3x_6 + x_3x_4x_5^2x_6, \quad 
x_1x_2^3x_6 + x_1x_3x_5^2x_6,\\
&x_1^3x_2x_6 + x_2x_4x_5^2x_6, \quad 
x_2x_3x_4x_5^2 + x_1x_2^2x_6^2, \quad 
x_1x_3x_4x_5^2 + x_1^2x_2x_6^2,\\
&x_2^2x_3x_4x_5 + x_1x_3x_5x_6^2, \quad 
x_1^2x_3x_4x_5 + x_2x_4x_5x_6^2,\quad 
x_2^3x_4x_5 + x_1x_2x_5x_6^2, \\
&x_1^3x_3x_5 + x_1x_2x_5x_6^2,\quad
x_1x_2^2x_3x_4 + x_1^2x_3x_6^2,\quad 
x_1^2x_2x_3x_4 + x_2^2x_4x_6^2,\\
&x_1x_2^3x_4 + x_1^2x_2x_6^2,\quad
x_1^3x_2x_3 + x_1x_2^2x_6^2,\quad 
x_1^2x_2^3 + x_1^2x_3x_5^2,\\ 
&x_1^3x_2^2 + x_2^2x_4x_5^2,\quad 
x_1^3x_3 + x_2^3x_4 + x_3x_4x_5^2 + x_1x_2x_6^2.
\end{aligned}
\]
By the inductive definition of $\fI{n}{f}$, the increasing sequence $\{\fI{n}{f}\}_{n \ge 1}$ stabilizes at $\fI{4}{f}$. 
Hence $\fI{n}{f} \subseteq \fI{4}{f} \subseteq \m^{[2]}$ for every integer $n \geq 1$. 
\end{example}

\begin{remark}\label{example-inv-adj}
The third author of this article proves a similar criterion for quasi-$F$-regularity (see \cite{Yoshikawa25-fedder}*{Theorem~A}). 
Applying this criterion, one can check that the ring $(S/(x_6,f))_\m$ in Example~\ref{example:fedder-qFs}(3) is quasi-$F$-regular (see \cite{Yoshikawa25-fedder}*{Example~1.2}). 
Consequently, quasi-$F$-regularity does not deform even for Gorenstein rings. 
\end{remark}

In general, it is difficult to compute the ideal $I_{n}(f)$ for large $n$. 
The following criterion is therefore useful for proving that a hypersurface singularity is not quasi-$F$-split in characteristic $p>2$.

\begin{proposition}[\cite{KTY22}*{Corollary~4.19}]\label{non-qfs}
With the notation as in Setting~\ref{notation:fedder-qFs}, if $f^{p-2} \in \m^{[p]}$, then $(S/f)_\m$ is not quasi-$F$-split.
\end{proposition}

\begin{example}\label{example:non-qFs}
We retain the notation of Setting~\ref{notation:fedder-qFs}.
Assume that $p \equiv m \pmod{N}$ for some $3 \leq m \leq N-1$, and set
\[
f=x_1^N+\cdots+x_N^N.
\]
Then we have
\[
f^{p-2} 
    \equiv 
    \binom{p-m}{\frac{p-m}{N},\ldots,\frac{p-m}{N}}
    (x_1\cdots x_N)^{p-m} f^{m-2}  
    \equiv 0 \pmod{\m^{[p]}}.
\]
Therefore, $(S/f)_\m$ is not quasi-$F$-split by Proposition~\ref{non-qfs}.

In the case $m=2$, the above argument does not apply, because $(x_1\cdots x_N)^{p-2}\notin \m^{[p]}$. 
In fact, in this case, the ring $(S/f)_\m$ is quasi-$F$-split if and only if $N=3$ by \cite{KTY22}*{Example~4.13} and \cite{KTY25}*{Example~3.17}. 
\end{example}

When the ring is the homogeneous coordinate ring of a Calabi-Yau complete intersection, we have a simpler criterion. 
\begin{theorem}[\cite{KTY22}*{Theorem~5.8}]\label{CY-fedder}
With the notation of Setting~\ref{notation:fedder-qFs}, suppose that $f_1, \dots, f_m \in S=k[x_1, \dots, x_N]$ form a homogeneous regular sequence and $f:=f_1 \cdots f_m$ has degree $N$.  
Define a sequence of elements $\{f^{(n)}\}_{n \ge 1}$ inductively by
\begin{align*}
f^{(1)}&:=f^{p-1},\\ 
f^{(n)}&:=f^{p-1}\Delta(f^{p-1})^{p^{n-2}+\cdots+1} \quad (n \ge 2).
\end{align*}
Then the following equalities hold: 
\[
\sht((S/(f_1,\ldots,f_m))_\m)=
\inf\{n \mid \theta_f^{n-1}(F^{n-1}_* f^{p-1}) \notin \m^{[p]} \}
=\inf\{n \mid f^{(n)} \notin \m^{[p^n]} \}.
\]
\end{theorem}

When $f \in S$ is a homogeneous polynomial and we consider the projective variety $\Proj(S/f)$, it is known that the quasi-$F$-split height \(\sht((S/f)_\m)\) agrees with its Artin–Mazur height (cf.~\cite{yobuko19} and \cite{KTY22}*{Theorem~5.8}). 
Theorem~\ref{CY-fedder} is particularly well suited to computer-assisted calculations and has found applications in related work (see \cite{KTY22}*{Example~6.2}, \cite{BGP} and \cite{GJ}). 
Furthermore, this enables us to construct examples of Calabi--Yau varieties with arbitrary Artin--Mazur height.

\begin{theorem}[\cite{KTY25}*{Theorem~4.18}]
For every prime number $p$ and every positive integer $h$, there exists a Calabi--Yau variety over $\mathbb{F}_p$ whose Artin--Mazur height is equal to $h$.
\end{theorem}

\subsection{Quasi-$F$-pure thresholds}

In this subsection, we introduce the notion of quasi-$F$-pure thresholds. 
With the notation as in Setting~\ref{notation:qFinj}, suppose that $\Delta=\operatorname{div}(f)$ is Cartier. 
If the pair $(A, t \Delta)$ is quasi-$F^{\infty}$-split, then so is $(A, s \Delta)$ for all $s \le t$. 
The quasi-$F$-pure threshold of $\Delta$ is the critical value of $t$ for which $(A, t \Delta)$ is quasi-$F^{\infty}$-split. 
This invariant can be viewed as a closer positive-characteristic analog of the lc threshold than the usual $F$-pure threshold. 
We describe a criterion for computing it. 

\begin{definition}\label{defn:threshold}
Let $(A,\m)$ be an $F$-finite normal local ring of characteristic $p>0$ and $f \in \m$ be a nonzero element.
Assume that $A$ is quasi-$F^{\infty}$-split.
The \emph{quasi-$F$-pure threshold} of $f$ is then defined as 
\[
\qfpt(A;f)
:= \sup \{\, c \in \Q_{\ge 0} \mid (A, c\,\mathrm{div}(f)) 
\text{ is quasi-$F^{\infty}$-split}\,\}.
\]
\end{definition}

\begin{setting}\label{notation:fedder-qFreg}
Let $A:=k[[x_1, \dots, x_N]]$ be the formal power series ring in $N$ variables over a perfect field $k$ of characteristic $p>0$ and $f \in (x_1, \dots, x_N)$ be a nonzero element of $A$. 
Let $\wt{A}:=W(k)[[x_1, \dots, x_N]]$ be the formal power series ring in $N$ variables over the ring $W(k)=\varprojlim_n W_n(k)$ of Witt vectors of $k$, and set $\m:=(x_1,\ldots,x_N)\wt{A}$ and $\m^{[p^n]}:=(x_1^{p^n},\ldots,x_N^{p^n})\wt{A}$ for every integer $n \ge 1$. 
Choose a lift $\wt{f} \in \m$ of $f$ such that $\wt{f}$ and $p$ form a regular sequence in $\wt{A}$. 
Let 
\[
F \colon W(k) \to W(k)\quad (a_0, a_1, \dots) \mapsto (a_0^p, a_1^p,\dots) 
\]
be the Frobenius map, which is a local ring isomorphism since $k$ is perfect. 
Let $\varphi \colon \wt{A} \to \wt{A}$ be the Frobenius lift, that is, the ring endomorphism extending $F$ and satisfying $\varphi(x_i)=x_i^p$ for all $1 \le i \le N$. 
Then for each integer $n \ge 1$, the set
\[
\{\phi^n_*(x_1^{e_1}\cdots x_N^{e_N}) \mid 0 \le e_1,\ldots,e_N \le p^n-1\}
\]
is a free $\wt{A}$-basis of $\phi^n_*\wt{A}$.
Let $\mathrm{Tr}^n \colon \varphi^n_*\wt{A} \to \wt{A}$ denote the $\wt{A}$-linear map that sends $\varphi^n_*((x_1 \cdots x_N)^{p^n-1})$ to 1 and all other basis elements to zero. 
\end{setting}

The following theorem generalizes Proposition~\ref{Fedder-fpt} and is useful for computing quasi-$F$-pure thresholds.

\begin{theorem}\label{comp-qfps}
We retain the notation of Setting~\ref{notation:fedder-qFreg}.
For each integer $e \ge 1$, we define $\nu^q_e(f)$ to be the largest  integer $s$ such that 
there exist an integer $n \ge 1$ and an element $g \in (\wt{f}^{sp^{\,n-1}}, p^n)\wt{A}$ satisfying the following conditions: 
\begin{enumerate}
    \item[(i)] $\mathrm{Tr}^{e+r-1}(\phi^{\,e+r-1}_* g) \in (p^r)$ for all $1 \le r \le n-1$, and
    \item[(ii)] $\mathrm{Tr}^{e+n-1}(\phi^{\,e+n-1}_* g)  \notin (\m, p^n)$.
\end{enumerate}
Then $\nu^q_e(f)$ is independent of the choice of the lift $\wt{f}$. 
Furthermore, $\left\{\nu^q_e(f)/{p^e}\right\}_{e \ge 1}$ is an increasing sequence, and
\[
\lim_{e \to \infty} \frac{\nu^q_e(f)}{p^e}
= \qfpt(A;f).
\]
\end{theorem}

\begin{proof}
Set $D := \mathrm{div}(f)$.  
It follows from \cite{Yoshikawa25-fedder}*{Theorem~6.9} that 
\begin{equation}\label{eq:nu}
\nu^q_e(f)
= \max\{\, s \in \Z_{\ge 0} \mid (A, \tfrac{s}{p^e}D) \text{ is quasi-$F^e$-split} \},
\end{equation}
and in particular,  $\nu^q_e(f)$ does not depend on the choice of the lift $\wt{f}$. 
Let $s := \nu^q_e(f)$, so $(A, \frac{s}{p^e}D)$ is quasi-$F^e$-split.
By \cite{Yoshikawa25-fedder}*{Proposition~6.4}, the pair 
$(A, \frac{s}{p^e}D)$ is also quasi-$F^{e+1}$-split.
Thus 
\[
\nu^q_{e+1}(f)\ge p s = p\,\nu^q_e(f), 
\]
and therefore $\left\{{\nu^q_e(f)}/{p^e}\right\}_{e \ge 1}$ is an increasing sequence.

Next, we prove that 
\[
\lim_{e \to \infty} \frac{\nu^q_e(f)}{p^e}
= \qfpt(A;f).
\]
We first show that $\lim_{e \to \infty} \nu^q_e(f)/p^e \ge \qfpt(A;f)$. 
Note that $\qfpt(A;f)>0$ because $A$ is quasi-$F$-regular and \cite{Yoshikawa25-fedder}*{Proposition~6.4}.  
Take $a \in \Q_{\ge 0}$ with $a < \qfpt(A;f)$ and $p^e a \in \Z$ for some $e \ge 1$.
Since $(A, aD)$ is quasi-$F^{\infty}$-split,  
\eqref{eq:nu} yields $\nu^q_e(f) \ge p^e a$.
Letting \(a \to \qfpt(A;f)\) from below gives
\[
\lim_{e \to \infty} \frac{\nu^q_e(f)}{p^e} \ge \qfpt(A;f).
\]

For the reverse inequality, fix $e \ge 1$.
By \eqref{eq:nu}, the pair $(A, \frac{\nu^q_e(f)}{p^e}D)$ is quasi-$F^e$-split, and therefore quasi-$F^{\infty}$-split by \cite{Yoshikawa25-fedder}*{Proposition 6.4} again. 
Thus
\[
\frac{\nu^q_e(f)}{p^e} \le \qfpt(A;f).
\]
Taking limits gives the desired inequality.
\end{proof}

Quasi-$F$-pure thresholds coincide with log canonical thresholds in dimension two. 

\begin{theorem}\label{qfpt-lct}
Let $(A,\m)$ be a $2$-dimensional $F$-finite normal local ring of characteristic $p>0$ and $D$ be an effective $\Q$-Weil divisor on $\Spec A$. 
If $\Spec A$ has klt singularities, then
\[
\qfpt(A;D)=\lct(A;D).
\]
\end{theorem}

\begin{proof}
Since $A$ is klt, we have 
\[
\lct(A;D)=\sup\{c \in \Q_{\ge 0} \mid (A,cD)\ \text{is klt}\,\}.
\]
Therefore, the assertion follows from Theorem~\ref{klt-to-qFreg}. 
\end{proof}

In Example~\ref{fpt example}, we saw that the $F$-pure threshold of $x^3+y^2$ depends strongly on the characteristic of the base field. 
For this particular polynomial, however, the quasi-$F$-pure threshold is independent of the characteristic. 

\begin{example}[cf.~\cite{Yoshikawa25-fedder}*{Example~6.10}]\label{cusp}
We use Setting~\ref{notation:fedder-qFreg}.
Let $N = 2$ and $f=x^3 + y^2$.  
Then
by Theorem~\ref{qfpt-lct}, we have
\[
\qfpt(A;f) = \lct(A;f) = \frac{5}{6}. 
\]

While Example~\ref{cusp} is an immediate consequence of Theorem~\ref{qfpt-lct}, we explain how to compute $\qfpt(A;f)$ using Theorem~\ref{comp-qfps}. 
Set $\wt{f} := x^3 + y^2 \in \wt{A}$, and assume for simplicity that $p \ne 2, 3$.   
We first show that $\qfpt(A;f) \le 5/6$.
Fix $e \ge 1$ and set $s=\nu^q_e(f)$. 
By Theorem~\ref{comp-qfps}, there exists an integer $n \ge 1$ such that
\[
\wt{f}^{sp^{n-1}} \notin (\m^{[p^{e+n-1}]},\, p^n).
\]
Thus there exist nonnegative integers $\alpha,\beta$ with $\alpha+\beta = sp^{n-1}$ such that
\[
(x^3)^\alpha (y^2)^\beta \notin \m^{[p^{e+n-1}]}=(x^{p^{e+n-1}}, y^{p^{e+n-1}}).
\]
It follows that $3\alpha < p^{e+n-1}$ and $2\beta < p^{e+n-1}$, and therefore $sp^{n-1} < ({5}/{6}) p^{e+n-1}$, i.e.\ $s/p^e<5/6$. 
Letting $e \to \infty$ yields 
\[
\qfpt(A;f)=\lim_{e \to \infty} \frac{\nu^q_e(f)}{p^e} \le \frac{5}{6}.
\] 

Next, we prove that $\qfpt(A;f) \ge 5/6$.
If $p \equiv 1 \pmod 6$, then Example~\ref{fpt example} gives  
\[\qfpt(A;f) \ge \fpt(A;f)= 5/6.\]
Thus we may assume that $p \equiv 5 \pmod 6$.
Fix an even integer $e \ge 2$ and set $t_e:=(5/6)(p^e - 1)$. 
We show that $t_e \le \nu^q_e(f)$.
Write 
\[
(\alpha_0,\beta_0) := \left(\frac{p^e - 1}{3}, \frac{p^e - 1}{2}\right),
\]
so that $t_e = \alpha_0 + \beta_0$.
Let $c$ be the coefficient of $(x^3)^{\alpha_0}(y^2)^{\beta_0}$ in the expansion of $\wt{f}^{t_e}$, and set $n$ to be the smallest positive integer such that $c \notin (p^n)$.

\begin{claim}\label{cusp:otherwise}
The integer $n$ defined above is equal to $e/2+1$. 
Furthermore, for $s:=t_e$, the element 
\[g:=\wt{f}^{s p^{n-1}}(xy)^{{p^{n-1}-1}}\] 
satisfies the conditions \textup{(1)} and \textup{(2)} of  Theorem~\ref{comp-qfps}. 
\end{claim}

\begin{proof}[Proof of Claim~\ref{cusp:otherwise}]
We repeatedly use Kummer’s theorem on binomial coefficients, which says that the $p$-adic order of $\binom{M+N}{N}$ equals the number of carries in the base-$p$ addition of $M$ and $N$.
We first show that $n = e/2 + 1$. 
The base-$p$ expansions of $(1/2)(p^e-1)$ and $(1/3)(p^e-1)$ are as follows: 
\begin{align*}
\frac{1}{2}(p^e - 1)
  &= \frac{p - 1}{2} + \frac{p - 1}{2}p + \cdots+\frac{p-1}{2}p^{e-2}+\frac{p-1}{2}p^{e-1},\\
\frac{1}{3}(p^e - 1)
  &= \frac{2p - 1}{3} + \frac{p - 2}{3}p  + \cdots+\frac{2p - 1}{3}p^{e-2}+\frac{p-2}{3}p^{e-1},
\end{align*}
Noting that
\begin{align*}
& \frac{p-1}{2}+\frac{2p-1}{3}=p+\frac{p-5}{6} \ge p,\\    
& \frac{p-1}{2}+\frac{p-2}{3}+1=\frac{5p-1}{6}<p, 
\end{align*}
we see from Kummer's theorem that the $p$-adic order of 
\[
c=\binom{t_e}{\beta_0}=\binom{\frac{5}{6}(p^e - 1)}{\frac{1}{3}(p^e - 1)}
\]
is $e/2$, so $n=e/2+1$. 

Fix $1 \le r \le n$.
Consider the terms in the expansion of $g$ whose exponents of $x$ and $y$ are both congruent to $-1$ modulo $p^{e+r-1}$.
Any such term has the form
\[
(x^3)^\alpha (y^2)^\beta {(xy)^{p^{n-1}-1}},
\]
where $\alpha,\beta \ge 0$ satisfy $\alpha + \beta = t_e p^{n-1}$ and
\begin{align*}
3\alpha + p^{n-1} - 1 &\equiv -1 \pmod{p^{e+r-1}}, \\
2\beta + p^{n-1} - 1 &\equiv -1 \pmod{p^{e+r-1}}.
\end{align*}
Since $p^{n-1}\mid 3\alpha$, $p^{n-1}\mid 2\beta$ and $p\ne 2,3$, it follows that $p^{n-1}\mid \alpha$ and $p^{n-1}\mid \beta$. 
Set $\alpha' := \alpha/p^{n-1}$ and $\beta' := \beta/p^{n-1}$.
Then $\alpha' + \beta' = t_e$ and
\[
3\alpha' \equiv 2\beta' \equiv -1 \pmod{p^{\,e/2 + r - 1}}.
\]

For $r \le n-1$, we examine the first $e/2+r-1$-digits of the base-$p$ expansions of $\alpha'$ and $\beta'$. 
Assume first that $e/2+r-1$ is odd. Then 
\begin{align*}
\alpha'
  &= \frac{2p - 1}{3} + \frac{p - 2}{3}p  + \cdots+\frac{2p - 1}{3}p^{e/2+r-2}+\cdots,\\
\beta'
  &= \frac{p - 1}{2} + \frac{p - 1}{2}p + \cdots + \frac{p - 1}{2}p^{e/2+r-2} + \cdots. 
\end{align*}
It follows from Kummer’s theorem that 
\[
v_p \left(\binom{t_e p^{n-1}}{\beta}\right)=v_p\left(\binom{t_e}{\beta'}\right) \ge \frac{1}{2}\left(\frac{e}{2}+r\right) \ge r, 
\]
where $v_p$ denotes the $p$-adic valuation. 
If $e/2+r-1$ is even, a similar argument still shows that the $p$-adic order of $\binom{t_e p^{n-1}}{\beta}$ is at least $r$. 
Thus, $\mathrm{Tr}^{e+r-1}(\phi^{e+r-1}_* g) \in (p^r)$. 

When $r=n$, we have $e/2+r-1=e$, and hence 
\[
3\alpha' \equiv 2\beta' \equiv -1 \pmod{p^{e}}.
\]
Since $0\le \alpha',\beta'\le t_e < p^e$, these congruences force
\[
\alpha'=\alpha_0=\frac{p^e-1}{3},
\quad
\beta'=\beta_0=\frac{p^e-1}{2}.
\]
Thus, among the terms in the expansion of $g$, the only monomial whose exponents of $x$ and $y$ are both congruent
to $-1$ modulo $p^{e+n-1}$ is the one corresponding to
$\alpha=\alpha_0p^{n-1}$ and $\beta=\beta_0p^{n-1}$.
Applying Kummer's theorem gives  
\[
v_p\left(\binom{t_e p^{n-1}}{\beta_0 p^{n-1}}\right)=v_p\left(\binom{t_e}{\beta_0}\right)=n-1,
\]
and therefore $\mathrm{Tr}^{e+n-1}(\phi^{e+n-1}_* g) \notin (\m, p^n)$. 
\end{proof}

It follows from Theorem~\ref{comp-qfps} and Claim~\ref{cusp:otherwise} that $t_e \le \nu^q_e(f)$ for all even integers $e$.
Consequently,
\[
\frac{5}{6}
  = \lim_{e \to \infty} \frac{t_e}{p^e}
  \le \lim_{e \to \infty} \frac{\nu^q_e(f)}{p^e}
  = \qfpt(A;f).
\]
\end{example}

\section{Open Problems}

We conclude this paper by listing a number of open problems on quasi-$F$-singularities. 
\begin{enumerate}
\item For every integer $e \ge 1$, are normal quasi-$F^e$-split singularities quasi-$F^e$-injective?
This is known when $e=1$ or when the ring $A$ is Cohen--Macaulay by Proposition~\ref{cor:qFinj-to-qF^einj}. 

\item Are quasi-$F$-regular singularities Cohen--Macaulay? 
This is known if the ring is $\Q$-Gorenstein of dimension three or smaller (see \cite{KTTWYY3}*{Theorem D}).

\item Let $(A,\m)$ be an $F$-finite regular local ring of characteristic $p>0$ and $f \in \m$ be a nonzero element.
Does $\qfpt(A;f)=1$ imply that $A/(f)$ is quasi-$F$-split? 
The converse holds by \cite{Yoshikawa25-fedder}*{Example~6.5}.

\item Let $A:=k[[x,y,z,w]]$ be the formal power series ring over a perfect field $k$ of characteristic $p>0$. 
For $f=x^4+y^4+z^4+w^4$, what is $\qfpt(A;f)$? 
It follows from Proposition~\ref{Fedder-fpt} that if $p \equiv 1 \pmod{4}$, then 
\[
    \qfpt(A;f) = \fpt(A;f) = 1.
\]
\end{enumerate}

\bibliographystyle{amsplain}
% (compat) commands used in amsrefs-style entries
\providecommand{\au}[1]{#1}
\providecommand{\at}[1]{#1}
\providecommand{\jt}[1]{#1}
\providecommand{\bvol}[1]{#1}
\providecommand{\yr}[1]{#1}
\providecommand{\st}[1]{#1}
\providecommand{\arXiv}{arXiv}
\providecommand{\arxiv}[1]{#1}

\end{document}